\documentclass[11pt]{article}

\usepackage{epic,latexsym,amssymb}
\usepackage{color}
\usepackage{tikz}
\usepackage{amsfonts,epsf,amsmath}
\usepackage{epic,latexsym,amssymb,xcolor}
\usepackage{xcolor}
\usepackage{tikz}
\usepackage{amsfonts,epsf,amsmath,leftidx}
\usepackage{float}
\usepackage{mathrsfs}
\usetikzlibrary{arrows}

\usepackage[bottom=1.5in, top=1.2in, left=1.5in, right=1.5in]{geometry}

\newtheorem{them}{Theorem}
\newtheorem{lem}{Lemma}
\newtheorem{quest}{Question}

\newtheorem{defn}{Definition}
\newtheorem{cl}{Claim}
\newtheorem{exm}{Example}
\newtheorem{cor}{Corollary}
\newtheorem{ob}{Observation}
\newtheorem{prop}{Proposition}
\newtheorem{conj}{Conjecture}
\newtheorem{rmk}{Remark}
\newtheorem{problem}{Problem}

\newcommand{\qed}{$\Box$}

\let\oldenumerate\enumerate
\renewcommand{\enumerate}{
  \oldenumerate
  \setlength{\itemsep}{0pt}
  \setlength{\parskip}{0pt}
  \setlength{\parsep}{0pt}
}

\begin{document}

\title{On irredundance coloring and irredundance compelling coloring of graphs}

\author{$^{\dagger}$David Ashok Kalarkop and $^{\dagger, \S}$Pawaton Kaemawichanurat\thanks{Research
supported by National Research Council of Thailand (NRCT) and King Mongkut's University of Technology Thonburi (N42A660926)}
\\ \\
$^{\dagger}$Department of Mathematics, Faculty of Science\\
King Mongkut's University of Technology Thonburi,\\
Thailand\\
$^{\S}$Mathematics and Statistics with Application (MaSA)
\\
\small \tt Email: david.ak123@gmail.com, pawaton.kae@kmutt.ac.th}

\date{}
\maketitle


%

\begin{abstract}
	Irredundance coloring of $G$ is a proper coloring in which there exists a maximal irredundant set $R$ such that all the vertices of $R$ have different colors. The minimum number of colors required for an irredundance coloring of $G$ is called the irredundance chromatic number of $G$, and is denoted by $\chi_{i}(G)$. Irredundance compelling coloring of $G$ is a proper coloring of $G$ in which every rainbow committee (the set containing a vertex of each color) is an irredundant set of $G$. The maximum number of colors required for an irredundance compelling coloring of $G$ is called the irredundance compelling chromatic number of $G$, and is denoted by $\chi_{irc}(G)$. In this paper, we make a detailed study on $\chi_{i}(G)$, $\chi_{irc}(G)$ and its relation to other coloring and domination parameters

\end{abstract}


\section{Introduction and Motivation}
Let $G$ denote a finite simple undirected connected graph with vertex set $V(G)$ and edge set $E(G)$. For $H \subseteq V(G)$, $G[H]$ denotes the subgraph induced by $H$. The \emph{open neighborhood} $N_{G}(v)$ of a vertex $v$ in $G$ is $\{u \in V(G) | uv \in E(G)\}$. Further, the \emph{closed neighborhood} $N_{G}[v]$ of a vertex $v$ in $G$ is $N_{G}(v) \cup \{v\}$. The \emph{degree} $deg_{G}(v)$ of a vertex $v$ in $G$ is $|N_{G}(v)|$. For a subset $S\subseteq V(G)$, we let $N_{G}(S) = \cup_{v \in S}N_{G}(v)$ and $N_{G}[S] = N_{G}(S) \cup S$. For a vertex $v \in S$, the \emph{private neighbor set} of $v$ with respect to $S$ is $pn[v,S] = N_{G}[v] - N_{G}[S - \{v\}]$. A vertex $u \in pn[v,S]$ is called a \emph{private neighbor} of $v$. The minimum cardinality of a vertex cut set of a graph $G$ is called the \emph{connectivity} and is denoted by $\kappa(G)$. If $G$ has $S = \{a\}$ as a minimum cut set, then $G$ contains $a$ as a \emph{cut vertex} and $\kappa(G) = 1$. Similarly, the minimum cardinality of an edge cut set of a graph $G$ is called the \emph{edge connectivity} and is denoted by $\kappa'(G)$. If $G$ has $L = \{e\}$ as a minimum edge cut set, then $G$ contains $e$ as a \emph{bridge} and $\kappa'(G) = 1$. We abbreviate $deg_{G}(v), N_{G}(v), N_{G}(S), N_{G}[v]$ and $N_{G}[S]$ to $deg(v), N(v), N(S), N[v]$ and $N[S]$, respectively. For basic graph theoretical definitions and terminologies, refer to \cite{EF}.
\vskip 5 pt

\indent Graph coloring, domination and irredundance in graphs are some of the important areas in graph theory and finds rich applications in many fields. A coloring (proper coloring) of a graph $G$ is the assignment of colors to the vertices of $G$ in such a way that any two adjacent vertices receive different colors. The minimum number of colors required for a coloring of $G$ is said to be the \textit{chromatic number} of $G$, denoted by $\chi(G)$. The coloring $\mathcal{C}=(V_1,V_2, \ldots, V_k)$ of $G$ partitions $V(G)$ into independent sets $V_i$ (for $1 \leq i \leq k$). The set $V_i$ is said to be the color class and we let $col(v)$ to denote the color of the vertex $v$. For more details on graph coloring, refer to \cite{CD}. A subset $D$ of $V(G)$ is a \emph{dominating set} of $G$ if every vertex in $V(G)-D$ is adjacent to a vertex in $D$. The minimum cardinality of a dominating set of $G$ is said to be the \textit{domination number} of $G$, denoted by $\gamma(G)$. For more details on domination in  graphs and its applications, refer to \cite{ST, UV}. A set $S \subseteq V(G)$ is said to be \textit{irredundant} if every vertex of $S$ has at least one private neighbor or equivalently $S$ is said to be irredundant if for each $v \in S$, either $v$ is isolated in $G[S]$ (subgraph induced on $S$) or $v$ has a private neighbor in $V-S$. A set $S$ is said to be a \textit{maximal irredundant set} of $G$ is $S \cup \{v\}$ is not an irredundant set of $G$, for every $v \in V(G)-S$. The minimum cardinality of a maximal irredundant set of $G$ is called the \textit{lower irredundance number}, and is denoted by $ir(G)$. For more details on irredundance in graphs, refer to \cite{ST}
\vskip 5 pt

\indent In recent years, many researchers have worked on problems involving both domination and coloring parameters like dominator coloring \cite {AB, GH, IJ, KL}, global dominator coloring \cite{QR, WX, YZ}, gamma coloring \cite{MN} , compelling coloring \cite{cd} and so on. Let $\mathcal{C}=(V_1,V_2, \ldots, V_k)$ be a coloring of $G$. A vertex $v$ is said to dominate color class $V_i$, if $v$ is adjacent to all the vertices of $V_i$ or $V_i =\{v\}$. A vertex $v$ is said to anti-dominate color class $V_i$, if $v$ is not adjacent to any vertex of $V_i$. A \textit{rainbow committee ($RC$)} is a set containing a vertex of each color. The \textit{Dominator coloring} was graphs was introduced by Gera et al. \cite{KL}. It is coloring of $G$ in which every vertex dominates a color class and minimum number of colors required for a dominator coloring of $G$ is called dominator chromatic number, and is denoted by $\chi_{d}(G)$. The \textit{Global dominator coloring} was graphs was introduced by Hamid and Rajeshwari \cite{QR}. It is the dominator coloring of $G$ in which every vertex anti-dominates a color class and minimum number of colors required for a global dominator coloring of $G$ is called global dominator chromatic number, and is denoted by $\chi_{gd}(G)$. The \textit{Gamma coloring} of graphs was introduced by Gnanaprakasam and Hamid \cite{MN}. It is the coloring in which there exists a dominating set $D$ of $G$ such that all the vertices of $D$ receive different colors. The minimum number of colors required for a gamma coloring of $G$ is said to be gamma chromatic number, and is denoted by $\chi_{\gamma}(G)$. The \textit{Compelling coloring} of graphs was introduced by Bachstein et al. \cite{cd}. It is the coloring of $G$ compelling property $\mathcal{P}$ if every $RC$ satisfies property $\mathcal{P}$. The $\mathcal{P}$-compelling chromatic number of $G$ is the minimum number of colors required for $\mathcal{P}$-compelling coloring of $G$, and is denoted by $\chi_{\mathcal{P}}(G)$. The authors proved that compelling coloring is a generalization of dominator coloring. The $\mathcal{P}$-compelling (where $\mathcal{P}$: dominating set of $G$) coloring of $G$ is a dominator coloring of $G$.
\vskip 5 pt

\indent Motivated by the results on the parameters involving coloring and domination, we make an attempt to study the problems involving coloring and irredundance. In this paper we introduce new concepts of colorings, called irredundance coloring of graphs (this parameter is different from irredundant coloring which was introduced by Arumugam et al. \cite{ab}) and irredundance compelling coloring of graphs in Section \ref{def}. The results of our study are presented in Sections \ref{ircoloring} and \ref{compelir} while open problems are given in Section \ref{problem}.
\vskip 5 pt
{ We shall use the following result.
	\begin{prop}\cite{ST}\label{intro}
		Every minimal dominating set in a graph $G$ is a maximal irredundant set of $G$.
	\end{prop}
}
 \section{Definitions and notations}\label{def}
In this section, we give definition of two new concept of graph coloring. The first of which was motivated by \cite{MN} while the second was motivated by \cite{cd}.
\vskip 5 pt

\subsection{Irredundance coloring}
\begin{defn}\label{def1}
 	A proper coloring of a graph $G$ is said to be an irredundance coloring of $G$ if there exists a maximal irredundant set $R$ such that all the vertices of $R$ receive different colors. The minimum number of colors required for an irredundance coloring of $G$ is called the irredundance chromatic number of $G$, and is denoted by $\chi_{i}(G)$.
\end{defn}
\vskip 5 pt

\indent By Definition \ref{def1}, we point out in the following remark that there always exists an irredundance coloring for all graphs.
\vskip 5 pt

 \begin{rmk}
 	Every graph $G$ admits an irredundance coloring. Consider a trivial coloring $\mathcal{C}$ of $G$ in which each vertex of $G$ is given a different color. Let $R$ be a maximal irredundant set of $G$. Then all the vertices of $R$ receive different colors in the coloring $\mathcal{C}$ and hence $\mathcal{C}$ is an irredundance coloring of $G$.
 \end{rmk}

 \begin{rmk}
 	Let $S \subseteq V(G)$. If all the vertices of $S$ receive different colors in the coloring $\mathcal{C}$ of $G$, then $S$ is said to be $\mathcal{C}$-colorful. Therefore the coloring $\mathcal{C}$ of $G$ is an irredundance coloring if there exists a maximal irredundant set $R$ of $G$ such that $R$ is $\mathcal{C}$-colorful.
 \end{rmk}
\vskip 5 pt


\indent The following examples establish $\chi_{i}(G)$ when the graph $G$ is bipartite. Interestingly, even there exists a tree $G$ in which $\chi_{i}(G) > \chi(G)$ as illustrated in Figure \ref{eexample1}.
\vskip 5 pt

 \begin{exm}\label{exm1}
 	Consider the complete bipartite graph $K_{m,n}$, where $m,n \geq 2$ and let $v_1v_2 \in E(K_{m,n})$. Then the set $\{v_1,v_2\}$ is colorful in any $\chi$-coloring of $K_{m,n}$. Therefore $\chi_{i}(K_{m,n})=\chi(K_{m,n})=2$.

 \end{exm}

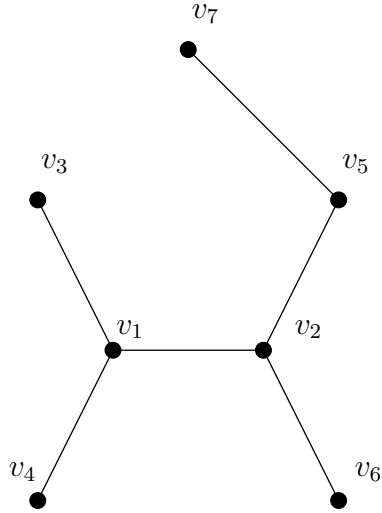
\begin{figure}[h!]
	\begin{center}
		\begin{tikzpicture}[line cap=round,line join=round,>=triangle 45,x=1cm,y=1cm]
			\clip(-4,-3.25) rectangle (5,5);
			\draw [line width=0.5pt] (-1,0)-- (1,0);
			\draw [line width=0.5pt] (-2,2)-- (-1,0);
			\draw [line width=0.5pt] (2,2)-- (1,0);
			\draw [line width=0.5pt] (-1,0)-- (-2,-2);
			\draw [line width=0.5pt] (1,0)-- (2,-2);
			\draw [line width=0.5pt] (2,2)-- (0,4);
			\begin{scriptsize}
				\draw [fill=black] (-2,2) circle (3pt);
				\draw[color=black] (-1.77,2.48) node {\large{$v_3$}};
				\draw [fill=black] (2,2) circle (3pt);
				\draw[color=black] (2.23,2.48) node {\large{$v_{5}$}};
				\draw [fill=black] (-1,0) circle (3pt);
				\draw[color=black] (-0.77,0.28) node {\large{$v_1$}};
				\draw [fill=black] (1,0) circle (3pt);
				\draw[color=black] (1.6,0.28) node {\large{$v_2$}};
				\draw [fill=black] (-2,-2) circle (3pt);
				\draw[color=black] (-2.2,-1.62) node {\large{$v_4$}};
				\draw [fill=black] (2,-2) circle (3pt);
				\draw[color=black] (2.4,-1.62) node {\large{$v_6$}};
				\draw [fill=black] (0,4) circle (3pt);
				\draw[color=black] (0.23,4.48) node {\large{$v_7$}};
			\end{scriptsize}
		\end{tikzpicture}
	\caption{A tree $G$ with $\chi_{i}(G)=3$ }
	\label{eexample1}
	\end{center}
\end{figure}

\vskip 15 pt

\begin{exm}
Consider the tree $G$ shown in Figure \ref{eexample1}. $\mathcal{C}=(\{v_1, v_5, v_6\}, \{v_2, v_3, v_4, v_7\})$ is the unique $\chi$-coloring of $G$. But $\mathcal{C}$ is not an irredundance coloring of $G$ since any set of the form $\{v_i\}$ or $\{v_i, v_j\}$ (where $col(v_i)\neq col(v_j)$) is not a maximal irredundant set. Therefore $\chi_{i}(G)\geq 3$. Now the coloring $\mathcal{C^\prime}=(\{v_1, v_6\}, \{v_2, v_3, v_4, v_7\}, \{v_5\})$ is an irredundance coloring of $G$ with the maximal irredundant set $\{v_1, v_2, v_5\}$ being $\mathcal{C^\prime}$-colorful. Hence $\chi_{i}(G)\leq 3$ which implies $\chi_{i}(G)= 3$.
\end{exm}
\vskip 5 pt

\indent We finish this subsection by pointing out an inequality chain between some chromatic numbers.
\vskip 5 pt

\begin{ob}
Every gamma coloring is an irredundance coloring. Let $\mathcal{C}$ be a gamma coloring of $G$. Then there exists a dominating set $D$ of $G$ such that $D$ is $\mathcal{C}$-colorful. Let $R$ be the minimal dominating set of $G$ contained in $D$. Then $R$ is maximal irredundant set of $G$ {(by Proposition \ref{intro})}. Since $R \subseteq D$, $R$ is $\mathcal{C}$-colorful. Thus $\mathcal{C}$ is an irredundance coloring of $G$. Hence $\chi_{i}(G) \leq \chi_{\gamma}(G)$. Also if $(V_1, V_2, \ldots, V_{\chi_d})$ is a dominator coloring of $G$, then the set $D=\{v_i: v_i \in V_i\}$ ($1 \leq i \leq \chi_d$) is a dominating set of $G$ (refer to \cite{IJ}). Thus by definition of gamma coloring, we have $\chi_{\gamma}(G) \leq \chi_{d}(G)$. Also $\chi_{d}(G) \leq \chi_{gd}(G)$ (refer to \cite{QR}). Thus we extend the chain of parameters involving coloring and domination as follows. For any graph $G$,
	$$\chi(G) \leq \chi_{i}(G) \leq \chi_{\gamma}(G) \leq \chi_{d}(G) \leq \chi_{gd}(G).$$
\end{ob}
\vskip 5pt

\subsection{Irredundance compelling coloring}
We now proceed to define irredundance compelling coloring of graphs.
\vskip 5 pt

\begin{defn}
A coloring of $G$ in which every rainbow committee ($RC$) is an irredundant set is called as irredundance compelling coloring ($IRC$) of $G$. The { maximum} number of colors required among all the IRC-colorings of $G$ is called as the irredundance compelling chromatic number, and is denoted by $\chi_{irc}(G)$. A graph which admits $IRC$-coloring is said to $IRC$-colorable.
\end{defn}
\vskip 5 pt

\indent This coloring does not always exists in all the graphs as detailed in the following observation.
\vskip 5 pt

\begin{ob}
 	Let $K_n$ be the complete graph. Then $\chi_{irc}(K_n)$ does not exists since $V(K_n)$ is not an irredundant set of $K_n$.
 \end{ob}
\vskip 5pt

\indent We now provide an application of $IRC$-coloring of graphs which will also serve as a motivation to study this parameter.
\vskip 5 pt

\indent Suppose we have a network of officers. This network can be represented by the graph $G$ with $V(G)$ being the set of officers and two vertices (officers) are adjacent if and only they know each other. Now there is a situation where each officer has a secret data that has to be secured. By the data being secured, we mean that if $S=\{F_1,F_2, \ldots, F_k\}$ be a group of officers. Suppose $F_i$ is not adjacent to any $F_j$, where $i\neq j$ and $1 \leq i,j \leq k$, there is no problem since the secret data of $F_i$ is not shared to any $F_j$. But consider the case when $F_i$ knows some $F_j$ in $S$. If $F_j$ is a friend of $F_i$, then there is no problem. But in case $F_j$ is an enemy of $F_i$, there is a possibility that the secret data of $F_i$ can be exposed in the set $S$. Therefore $F_i$ always makes sure that he has at least one close friend $CF$ in $V(G)-S$ such that $CF$ does not know anyone in $S-\{F_i\}$ so that $F_i$ can send the data to $CF$ and protect it from the set $S$.

Now we shall partition $G$ into independent sets $V_1,V_2, \ldots ,V_s$. Since any two officers in $V_l$ ($1 \leq l \leq s$) do not know each other,  their secret data is secured. Suppose there comes a situation where a group of officers $D$ has to be formed such that $D=\{F_l : F_l \in V_l\}$, where officer $F_l$ serves has a representative of $V_l$. This raises the following question.
\begin{quest} \label{qq1}
	Can $V(G)$ be partitioned into independent sets $V_1,V_2, \ldots ,V_s$ such that every officer $F_l$ keeps his data secured in any set $D=\{F_l : F_l \in V_l\}$.
\end{quest}
\vskip 5 pt

\noindent The answer to the Question \ref{qq1} is YES if and only if the network $G$ is $IRC$-colorable.
\vskip 5 pt

\indent Now the next question that arises is that whether there is a practical scenario in real world where the situation of Question \ref{qq1} arises. So we now give a practical scenario. Consider the sets $V_1,V_2, \ldots ,V_s$, where each $V_l$ represents a group of special agents(having a secret data) of a particular country. Note that any two agents of the same country do not know each other since they belong to different organizations and work independently. But two agents from different countries can have a connection. Suppose every year, a group $D$ of agents of different countries is selected to meet and discuss over some issues (selection of agents is random). Then every agent in the set $D$ secures his data if and only if the network of agents is $IRC$-colorable.

\section{On irredundance chromatic number of graphs}\label{ircoloring}

In this section, we present some results of our study on the irredundance chromatic number of graphs. We begin with a condition that implies $\chi_{i}(G) = \chi(G)$.
\vskip 5 pt

\begin{prop}\label{irp1}
	If $G$ contains a full degree vertex, then $\chi_{i}(G)=\chi(G)$.
\end{prop}
\begin{proof}
	Let $v$ be the full degree vertex of $G$ and $\mathcal{C}$ be any $\chi$-coloring of $G$. We shall prove that $\{v\}$ is maximal irredundant set of $G$. Suppose not, then there is some vertex $w$ such that $\{v,w\}$ is irredundant set of $G$. But $w$ is adjacent to $v$ and every neighbor of $w$ (if $deg(w)\geq 2$) is also adjacent to $v$. This implies that $w$ has no private neighbor with respect to the set $\{v,w\}$, which is a contradiction. Thus $\{v\}$ is maximal irredundant set of $G$ and $\mathcal{C}$-colorful. Hence $\chi_{i}(G)=\chi(G)$. \qed
\end{proof}
\vskip 5 pt

\noindent Proposition \ref{irp1} implies the following corollaries.
\vskip 5 pt

\begin{cor}
	Let $K_n$ be a complete graph of order $n$. Then $\chi_{i}(K_n)=n$.
\end{cor}

\begin{cor}
	Let $K_{1,n-1}$ be a star of order $n$. Then $\chi_{i}(K_{1,n-1})=2$.
\end{cor}

\begin{rmk}
	The converse of Proposition \ref{irp1} may not be true. Refer Example \ref{exm1} for a counter example.
\end{rmk}
\vskip 5pt

\indent In the following, we establish upper and lower bounds of $\chi_{i}(G)$ in terms of $\chi(G)$ and $ir(G)$, Then, the study of extremal graphs satisfying these bounds are presented.
\vskip 5 pt

\begin{them}\label{thmchi}
For any graph $G$, we have $$max\{\chi(G), ir(G)\} \leq \chi_{i}(G) \leq \chi(G)+ir(G)-1.$$
The bounds are sharp.
\end{them}
\begin{proof}
	By the definition of $\chi_{i}(G)$, it is clear that $\chi_{i}(G) \geq \chi(G)$. Suppose $\chi_{i}(G) < ir(G)$, then there exists a maximal irredundant set $R$ such that $|R| < ir(G) $, which is a contradiction. Thus $\chi_{i}(G) \geq ir(G)$ and hence the lower bound follows. Now for the upper bound, consider the $\chi$-coloring $\mathcal{C}$ of $G$. Let $R$ be a maximal irredundant set of minimum cardinality (i.e. $|R|=ir(G)$). Now the new coloring $\mathcal{C^\prime}$ is the modification of the coloring $\mathcal{C}$ such that exactly $ir(G)-1$ number of vertices in $R$ are given new colors (other than the colors used in $\mathcal{C}$). In the coloring $\mathcal{C^\prime}$, $R$ is $\mathcal{C^\prime}$-colorful. Thus $\chi_{i}(G) \leq \chi(G)+ir(G)-1$. Some trivial examples of extremal graphs of is a complete bipartite graph $K_{n, m}$ when the graph achieves the lower bound while is a complete graphs $K_{n}$ when the graph achieves the upper bound. \qed
\end{proof}
\vskip 5 pt

\indent We further study non-trivial extremal graphs satisfying the bounds in Theorem \ref{thmchi}. We can show that there infinitely many graphs $G$ achieving the lower bound.
\vskip 5 pt

\noindent For $n, k \in \mathbb{N}$ such that $n \geq 2k$, we let
\vskip 5 pt

\indent $\mathcal{A}(n, k)$ be the family of the graphs $G$ order $n$ such that $\chi(G)=ir(G)=\chi_{i}(G)=k.$
\vskip 5 pt

\noindent Clearly, all graphs $G$ in the family $\mathcal{A}(n, k)$ meets the lower bound when $ir(G) = \chi(G) = k$. The following lemma shows that this family is non-empty.
\vskip 5 pt

\begin{lem}\label{ll1}
For $n, k \in \mathbb{N}$ such that $n \geq 2k$, we have that $\mathcal{A}(n, k) \neq \emptyset$.
\end{lem}
\begin{proof} For a given $k \in \mathbb{N}$, we let $K_k \circ K_1$ be the corona of the complete graph $K_k$ and we let $v \in V(K_k)$. Then, for $n \in \mathcal{N}$ in which $n \geq 2k$, we let $G(n, k)$ be obtained from $K_k \circ K_1$ and the vertices $u_1, u_2, \ldots, u_{2k-n}$ by joining $u_1, u_2, \ldots, u_{2k-n}$ to $v$. It is clear that $\chi(G(n, k))=k$. Also $V(K_k)$ is a maximal irredundant set of $G(n, k)$ of minumum cardinality, hence $ir(G(n, k))=k$. In any $\chi$-coloring of $G(n, k)$, $V(K_k)$ is colorful and hence $\chi_{i}(G(n, k)) = k$. Thus, $G(n, k) \in \mathcal{A}(n, k)$, establishing the proof.\qed
\end{proof}
\vskip 5 pt
\indent We now give a construction of graphs attaining the upper bound of Theorem \ref{thmchi}, i.e. the graphs $G$ with $\chi_{i}(G)=\chi(G)+ir(G)-1$.
\vskip 5 pt

\noindent For integers $k$ and $l$ such that $k \geq 3$, we let
\vskip 5 pt

\indent $\mathcal{Z}(k, l)$ be the family of the graphs $G$ such that $\chi(G)=k$, $ir(G)=l$ and $\chi_{i}(G)=k+l-1$.
\vskip 5 pt

\begin{figure}[h!]
	\begin{center}
		
		\begin{tikzpicture}[line cap=round,line join=round,>=triangle 45,x=1.5cm,y=1.5cm]
			\clip(-3.35,-3.19) rectangle (5,4);
			\draw [line width=0.5pt] (0,3)-- (0,0);
			\draw [line width=0.5pt] (0,3)-- (3,3);
			\draw [line width=0.5pt] (3,3)-- (3,0);
			\draw [line width=0.5pt] (3,0)-- (0,0);
			\draw [line width=0.5pt] (0,0)-- (3,3);
			\draw [line width=0.5pt] (1.45,0.79)-- (0,0);
			\draw [line width=0.5pt] (1.45,0.79)-- (3,0);
			\draw [line width=0.5pt] (0,0)-- (-3,0);
			\draw [line width=0.5pt] (0,0)-- (-2.83,-1.05);
			\draw [line width=0.5pt] (0,0)-- (0,-3);
			\begin{scriptsize}
				\draw [fill=black] (0,0) circle (2pt);
				\draw[color=black] (0.17,0.3) node {\large{$v$}};
				\draw [fill=black] (0,3) circle (2pt);
				\draw[color=black] (0.17,3.32) node {\large{$a$}};
				\draw [fill=black] (3,3) circle (8pt);
				\draw[color=black] (3.17,3.42) node {\large{$K_{k-2}$}};
				\draw [fill=black] (3,0) circle (2pt);
				\draw[color=black] (3.17,0.32) node {\large{$u$}};
				\draw [fill=black] (-3,0) circle (2pt);
				\draw[color=black] (-2.77,0.38) node {\large{$p_1$}};
				\draw [fill=black] (-2.83,-1.05) circle (2pt);
				\draw[color=black] (-2.61,-0.66) node {\large{$p_2$}};
				\draw [fill=black] (0,-3) circle (2pt);
				\draw[color=black] (0.55,-2.62) node {\large{$p_{k+l-2}$}};
				\draw [fill=black] (-2.05,-1.87) circle (0.7pt);
				\draw [fill=black] (-1.69,-2.23) circle (0.7pt);
				\draw [fill=black] (-1.37,-2.47) circle (0.7pt);
				\draw [fill=black] (1.45,0.79) circle (2pt);
				\draw[color=black] (1.61,1.1) node {\large{$b$}};
			\end{scriptsize}
		\end{tikzpicture}
		\caption{The graph $H$}\label{H}
	\end{center}
\end{figure}
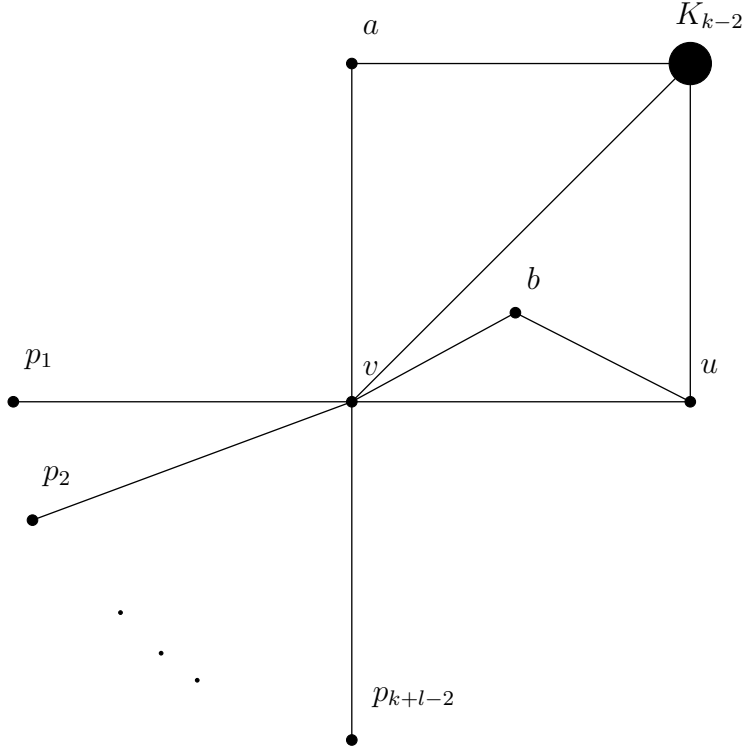


\indent We will describe a graph operation that we will be used in the next result. Let $H$ be an induced subgraph of a graph $G$. By the identification (or merging) of $H$ on two copies of $G$, we mean that the vertex $v \in V(H)$ of first copy of $G$ is identified with the vertex $v \in V(H)$ of second copy of $G$ to get new graph which we denote by ${[G]}_{H}^{2}$. Note that if $v_1v_2$ is an edge in $H$, then there will be two edges between $v_1$ and $v_2$ in ${[G]}_{H}^{2}$. Since an extra edge is of no use to us, so we remove all the edges except one between $v_1$ and $v_2$ in ${[G]}_{H}^{2}$. The same process is involved for constructing ${[G]}_{H}^{n}$, where $H$ is merged over $n$-copies of $G$.

\begin{them}
	For every positive integers $k \geq 3$ and $l$, $\mathcal{Z}(k, l) \neq \emptyset$.
\end{them}
\begin{proof}
In the proof, it suffices to show that, for every positive integers $k \geq 3$ and $l$, there exists a graph $G$ with $\chi(G)=k$, $ir(G)=l$ and $\chi_{i}(G)=k+l-1$. Let $H$ be the graph as shown in Figure \ref{H} with $V(H)=\{v, u, a, b, p_j: 1 \leq j \leq k+l-2\} \cup V(K_{k-2})$ and the edge set is described in Figure \ref{H} (note that every vertex of $V(K_{k-2})$ is joined to $a, v$ and $u$). Let $S=V(K_{k-2}) \cup \{u\} \subset V(H)$. Let $G={[H]}_S^{l}$ and for $1 \leq i \leq l$, let $v^i, u^i,a^i,b^i, p_{j}^{i}$ represent the vertices in the $i^{th}$-copy the $H$. The graph $[H]_S^2$ is shown in Figure \ref{ube}. We shall prove that $G$ is the required graph. We shall prove that $\chi(G)=k$. Clearly $\chi(G)\geq k$ since the induced graph $G[V(K_{k-2}) \cup \{v,a\}]$ is a clique of order $k$. Let $V(K_{k-2})=\{v_1, v_2, \ldots, v_{k-2}\}$. The coloring $\mathcal{C}=(\{v_1,p_{j}^{i}: 1 \leq j \leq k+l-2\}, \{v_2, b^i\}, \{v_3\}, \{v_4\}, \ldots, \{v_{k-2}\} \{a^i,u^i: 1 \leq i \leq l\}, \{v^i: 1 \leq i \leq l\})$ is a proper coloring of $G$. Thus $\chi(G)\leq k$. \\
	We shall prove the following claim which will help us to prove that $ir(G)=l$.
	\begin{cl}\label{ube11}
		Any maximal irredundant set of $G$ either contains $v^i$ or $\{p_j^i: 1 \leq j \leq k+l-2\}$, for all $1 \leq i \leq l$.
	\end{cl}
	It is very clear that if some maximal irredundant set of $G$ contains a pendant attached to $v^i$, then it contains all the pendants attached to $v^i$. Suppose there exists a maximal irredundant set $R$ of $G$ such that $v^i \notin R$ and $p_j^i \notin R$, for some $i$, say for $i=1$. If $v^1$ is not the private neighbor of any vertex of $R$, then $R \cup \{p_1^1, p_2^1, \ldots, p_{k+l-2}^1\}$ is a irredundant set of $G$, contradicting the maximality of $R$. Now suppose $v^1$ is the only external private neighbor ($epn$) of some vertex of $R$. If $v^1$ is the only $epn$ of $v_1 \in V(K_{k-2})$. Then the vertices $a^1, b^1$ and $u^1$ does not belong to $R$ (since $R$ is irredundant). This implies that $v_1$ has $epn$ $a^1$ (which is other than $v^1$) with respect to $R$. Suppose that $v^1$ is the only $epn$ of $a^1 \in R$. Then some vertex of $V_{k-2}$ has to belong to $R$, say $v_1$. But $a^1$ has no private neighbor with respect to $R$, contradicting that $R$ is irredundant set of $G$. Similarly $v^1$ cannot be the only $epn$ of $b^1$ or $u^1$. Thus $R \cup \{p_1^1, p_2^1, \ldots, p_{k+l-2}^1\}$ is a irredundant set of $G$, contradicting the maximality of $R$. Thus every maximal irredudant set of $G$ either contains $v^i$ or $\{p_j^i: 1 \leq j \leq k+l-2\}$, for all $1 \leq i \leq l$. Thus the claim is proved.
	\vskip 5pt
	
	From Claim \ref{ube11}, it is clear if $R$ is an $ir$-set of $G$, then $|R| \geq l$. Thus  $ir(G) \geq l$. Let $S=\{v^i: 1 \leq i \leq l\}$. Clearly $S$ is an $\gamma$-set of $G$ and hence a maximal irredundant set of $G$. Thus $ir(G) \leq l$. Now we shall proceed to prove that $\chi_{i}(G)=k+l-1$. Now the new coloring $\mathcal{C}^{\prime}=(\mathcal{C} \cup \{v^i\})$ (for $2 \leq i \leq l$) is an irredundance coloring of $G$ in which an $ir$-set $S$ is colorful. Thus $\chi_{i}(G) \leq k+l-1$. By Claim \ref{ube11}, if $R \neq S$ is a maximal irredundant set of $G$, then for some $i$, say $i=1$, the set $R_1=\{p_1^1, p_2^1, \ldots , p_{k+l-2}^1\}$ is contained in $R$. Since at least $k$ colors are used to color $G$, one of the $k-1$  colors used may be used to color some vertices in $R_1$ (since the color used for the vertex adjacent to the pendents cannot be used to color the pendents). But even after that since $R$ has to be colorful in an irredundance coloring, at least $l-1$ vertices in $R_1$ have to receive new colors and thus $\chi_{i}(G) \geq k+l-1$. \qed
\end{proof}
\vskip 15 pt

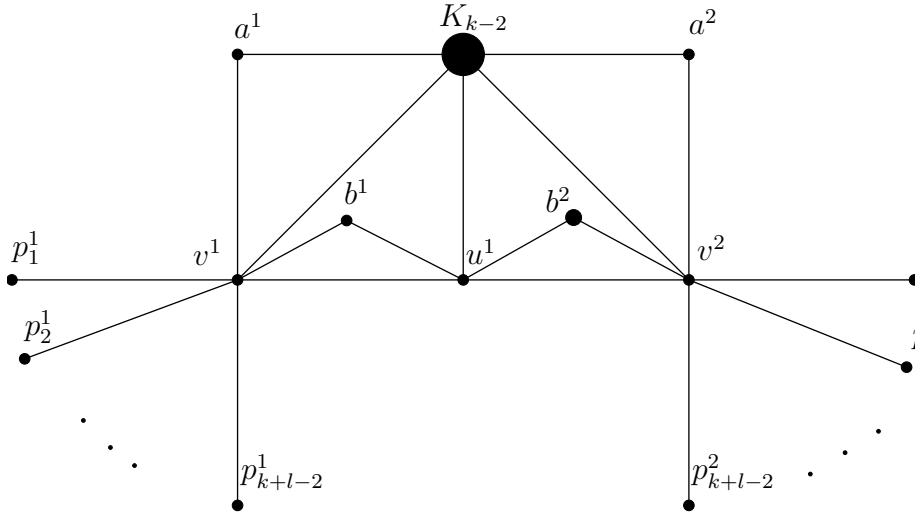
\begin{figure}[h!]
	\begin{center}
		
		\begin{tikzpicture}[line cap=round,line join=round,>=triangle 45,x=1cm,y=1cm]
			\clip(-3.06273751259392,-3.560728486345449) rectangle (9,4);
			\draw [line width=0.5pt] (0,3)-- (0,0);
			\draw [line width=0.5pt] (0,3)-- (3,3);
			\draw [line width=0.5pt] (3,3)-- (3,0);
			\draw [line width=0.5pt] (3,0)-- (0,0);
			\draw [line width=0.5pt] (0,0)-- (3,3);
			\draw [line width=0.5pt] (1.45,0.79)-- (0,0);
			\draw [line width=0.5pt] (1.45,0.79)-- (3,0);
			\draw [line width=0.5pt] (0,0)-- (-3,0);
			\draw [line width=0.5pt] (0,0)-- (-2.83,-1.05);
			\draw [line width=0.5pt] (0,0)-- (0,-3);
			\draw [line width=0.5pt] (3,3)-- (6,3);
			\draw [line width=0.5pt] (6,3)-- (6,0);
			\draw [line width=0.5pt] (6,0)-- (3,0);
			\draw [line width=0.5pt] (4.467633708795116,0.8299410626889091)-- (3,0);
			\draw [line width=0.5pt] (4.467633708795116,0.8299410626889091)-- (6,0);
			\draw [line width=0.5pt] (6,0)-- (6,-3);
			\draw [line width=0.5pt] (6,0)-- (8.894067432756044,-1.1610652709802607);
			\draw [line width=0.5pt] (6,0)-- (9,0);
			\draw [line width=0.5pt] (3,3)-- (6,0);
			\begin{scriptsize}
				\draw [fill=black] (0,0) circle (2pt);
				\draw[color=black] (-0.4,0.37107790040406827) node {\large{$v^1$}};
				\draw [fill=black] (0,3) circle (2pt);
				\draw[color=black] (0.14786103842360682,3.3809179277025327) node {\large{$a^1$}};
				\draw [fill=black] (3,3) circle (8pt);
				\draw[color=black] (3.134370538927364,3.48) node {\large{$K_{k-2}$}};
				\draw [fill=black] (3,0) circle (2pt);
				\draw[color=black] (3.234370538927364,0.37663158493387494) node {\large{$u^1$}};
				\draw [fill=black] (-3,0) circle (2pt);
				\draw[color=black] (-2.8030947775503443,0.4299621117285848) node {\large{$p_1^1$}};
				\draw [fill=black] (-2.83,-1.05) circle (2pt);
				\draw[color=black] (-2.6431031971662144,-0.61887158190071) node {\large{$p_2^1$}};
				\draw [fill=black] (0,-3) circle (2pt);
				\draw[color=black] (0.59673683013225325,-2.55654738877517) node {\large{$p_{k+l-2}^1$}};
				\draw [fill=black] (-2.05,-1.87) circle (0.7pt);
				\draw [fill=black] (-1.69,-2.23) circle (0.7pt);
				\draw [fill=black] (-1.37,-2.47) circle (0.7pt);
				\draw [fill=black] (1.45,0.79) circle (2pt);
				\draw[color=black] (1.5877852618807757,1.1765894868545235) node {\large{$b^1$}};
				\draw [fill=black] (6,3) circle (2pt);
				\draw[color=black] (6.191987408490736,3.4342484544972427) node {\large{$a^2$}};
				\draw [fill=black] (6,0) circle (2pt);
				\draw[color=black] (6.3,0.4299621117285848) node {\large{$v^2$}};
				\draw [fill=black] (4.467633708795116,0.8299410626889091) circle (3pt);
				\draw[color=black] (4.2631789737090505,1.06547369817904) node {\large{$b^2$}};
				\draw [fill=black] (9,0) circle (2pt);
				\draw[color=black] (9.26738112031901,0.4299621117285848) node {\large{$p_1^2$}};
				\draw [fill=black] (8.894067432756044,-1.1610652709802607) circle (2pt);
				\draw[color=black] (9.16072006672959,-0.7255326354901298) node {\large{$p_2^2$}};
				\draw [fill=black] (6,-3) circle (2pt);
				\draw[color=black] (6.591987408490736,-2.55654738877517) node {\large{$p_{k+l-2}^2$}};
				\draw [fill=black] (8.520753745193074,-2.014353699695619) circle (0.7pt);
				\draw [fill=black] (8.076332688570492,-2.298783175934072) circle (0.7pt);
				\draw [fill=black] (7.614134789683005,-2.583212652172525) circle (0.7pt);
			\end{scriptsize}
		\end{tikzpicture}
		\caption{The graph $G=[H]_S^2$}\label{ube}
	\end{center}
\end{figure}

\vskip 5pt

\indent Next, we establish upper and lower bounds of $\chi_{i}(G)$ in terms of the order. We completely characterize all the extremal graphs. Before stating the theorem, let us introduce some graphs properties and constructions.
\vskip 5 pt

\indent A bipartite graph $G(V_1,V_2)$ has \emph{Property} $\Im_1$ if there exists two adjacent vertices $v_1 \in V_1$ and $v_2 \in V_2$ with $deg(v_i) \geq 2$, for $i= 1, 2$ such that the following conditions are satisfied:
\begin{enumerate}
	\item Every vertex in $ V_2-N(v_1)$ is adjacent to every vertex in $N(v_2)-\{v_1\}$.
	\item Every vertex in $V_1-N(v_2)$ is adjacent to every $N(v_1)-\{v_2\}$.
	\item If $deg(v_i) \geq 3$, then for every $x \in N(v_i)$, either $x$ is a pendant vertex or $N(x) \subseteq N(v_j)$ or $N(x)=V_i$ (where $i,j=1,2$ , $i \neq j$).
	\end{enumerate}
The bipartite graph $G(V_1,V_2)$ has \emph{Property} $\Im_2$ if there exists non-adjacent vertices $v_1 \in V_1$ and $v_2 \in V_2$ such that $N(v_1)=V_2-\{v_2\}$ and $N(v_2)=V_1-\{v_1\}$. An example of graphs in $G_1$ that has Property $\Im_1$ and $G_2$ that has Property $\Im_2$ is shown in Figure \ref{example2}.

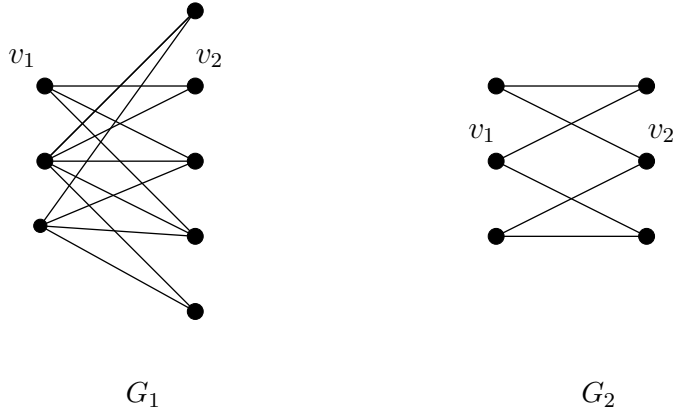
\begin{figure}[h!]
	\begin{center}
	\begin{tikzpicture}[line cap=round,line join=round,>=triangle 45,x=1cm,y=1cm]
		\clip(-7,-1.5) rectangle (5,5);
		\draw [line width=0.5pt] (-5,3)-- (-3,3);
		\draw [line width=0.5pt] (-5,3)-- (-3,2);
		\draw [line width=0.5pt] (-5,3)-- (-3,1);
		\draw [line width=0.5pt] (-5,2)-- (-3,3);
		\draw [line width=0.5pt] (-5,2)-- (-3,2);
		\draw [line width=0.5pt] (-5,2)-- (-3,1);
		\draw [line width=0.5pt] (-5,2)-- (-3,0);
		\draw [line width=0.5pt] (-5.06,1.14)-- (-3,2);
		\draw [line width=0.5pt] (-5.06,1.14)-- (-3,1);
		\draw [line width=0.5pt] (-5.06,1.14)-- (-3,0);
		\draw [line width=0.5pt] (-3,4)-- (-5,2);
		\draw [line width=0.5pt] (-5,2)-- (-3,4);
		\draw [line width=0.5pt] (-3,4)-- (-5.06,1.14);
		\draw [line width=0.5pt] (1,3)-- (3,3);
		\draw [line width=0.5pt] (3,1)-- (1,1);
		\draw (1.9976600207364084,-0.8076917977183633) node[anchor=north west] {\textbf{$G_2$}};
		\draw (-4.05667075932149,-0.8076917977183633) node[anchor=north west] {\textbf{$G_1$}};
		\draw [line width=0.5pt] (1,3)-- (3,2);
		\draw [line width=0.5pt] (3,2)-- (1,1);
		\draw [line width=0.5pt] (3,3)-- (1,2);
		\draw [line width=0.5pt] (1,2)-- (3,1);
		\begin{scriptsize}
			\draw [fill=black] (-5,3) circle (3pt);
			\draw[color=black] (-5.3,3.37) node {\large{$v_{1}$}};
			\draw[color=black] (0.813462106828727,2.373771506185263) node {\large{$v_{1}$}};
			\draw [fill=black] (-5,2) circle (3pt);
			\draw [fill=black] (-5.06,1.14) circle (2.5pt);
			\draw [fill=black] (-3,3) circle (3pt);
			\draw[color=black] (-2.8106405609004828,3.37) node {\large{$v_{2}$}};
			\draw[color=black] (3.2,2.373771506185263) node {\large{$v_{2}$}};
			\draw [fill=black] (-3,2) circle (3pt);
			\draw [fill=black] (-3,1) circle (3pt);
			\draw [fill=black] (-3,0) circle (3pt);
			\draw [fill=black] (1,3) circle (3pt);
			\draw [fill=black] (3,3) circle (3pt);
			\draw [fill=black] (1,1) circle (3pt);
			\draw [fill=black] (3,1) circle (3pt);
			\draw [fill=black] (-3,4) circle (3pt);
			\draw [fill=black] (13.257228912415627,-4.396209549066939) circle (2.5pt);
			\draw [fill=black] (1,2) circle (3pt);
			\draw [fill=black] (3,2) circle (3pt);
		\end{scriptsize}
	\end{tikzpicture}
		\caption{ Graphs $G_1 \in \Im_1$ and $G_2 \in \Im_2$ }
		\label{example2}
		
	\end{center}
\end{figure}
\vskip 15 pt

\indent We, further, give the construction of graphs $G$ with $\chi_i(G)=2$. Let $K$ be obtained from two stars centered at $v_{1}$ and $x$ and a vertex $v_{2}$ by joining $v_{2}$ to $v_{1}$ and $x$. The graph $K$ is shown in Figure \ref{newexample}. Note that $deg(v_1) \geq 2$ and $deg(x) \geq 2$. A bipartite graph $G$ belongs to the family $\mathcal{F}_1$ if $G$ is obtained from $K$ by joining some vertices(possibly none) or all the vertices of $N(v_1)$ to $x$. A bipartite graph $G(V_1,V_2)$ is said to belong to the family $\mathcal{F}_2$ if there exists two vertices $v_1 \in V_1$ and $v_2 \in V_2$ such that $N(v_1)=V_2$ and $N(v_2)=V_1$. A bipartite graph $G(V_1,V_2)$ is said to belong to the family $\mathcal{F}_3$ if there exists two non-adjacent vertices $v_1 \in V_1$ and $v_2 \in V_2$ such that $N(v_1)=V_2-\{v_2\}$ and $N(v_2)=V_1-\{v_1\}$.
\vskip 5 pt

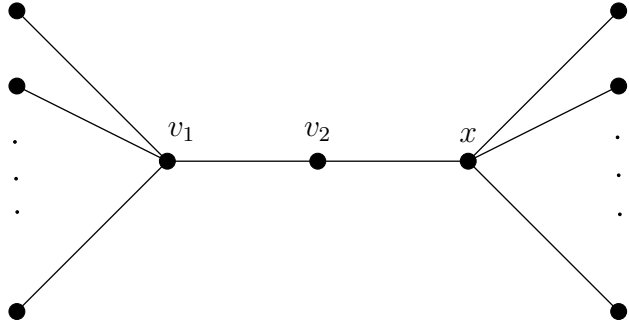
\begin{figure}[h]
	\begin{center}
\begin{tikzpicture}[line cap=round,line join=round,>=triangle 45,x=1cm,y=1cm]
	\clip(-5,-2) rectangle (5,5);
	\draw [line width=0.5pt] (-2,1)-- (-4,3);
	\draw [line width=0.5pt] (-2,1)-- (-4,2);
	\draw [line width=0.5pt] (-2,1)-- (-4,-1);
	\draw [line width=0.5pt] (-2,1)-- (0,1);
	\draw [line width=0.5pt] (0,1)-- (2,1);
	\draw [line width=0.5pt] (2,1)-- (4,3);
	\draw [line width=0.5pt] (2,1)-- (4,2);
	\draw [line width=0.5pt] (2,1)-- (4,-1);
	\begin{scriptsize}
		\draw [fill=black] (-2,1) circle (3pt);
		\draw[color=black] (-1.8119868357797562,1.393541363342264) node {\large{$v_1$}};
		\draw [fill=black] (0,1) circle (3pt);
		\draw[color=black] (0,1.393541363342264) node {\large{$v_2$}};
		\draw [fill=black] (2,1) circle (3pt);
		\draw[color=black] (2,1.3443736715145511) node {\large{$x$}};
		\draw [fill=black] (4,3) circle (3pt);
		\draw [fill=black] (4,2) circle (3pt);
		\draw [fill=black] (4,-1) circle (3pt);
		\draw [fill=black] (-4,3) circle (3pt);
		\draw [fill=black] (-4,2) circle (3pt);
		\draw [fill=black] (-4,-1) circle (3pt);
		\draw [fill=black] (-4.026093331139379,1.256284604574103) circle (0.5pt);
		\draw [fill=black] (-4.0108046170483265,0.7670457536603332) circle (0.5pt);
		\draw [fill=black] (-3.995515902957271,0.3236730450197293) circle (0.5pt);
		\draw [fill=black] (3.9851928525735953,1.3174394609383233) circle (0.5pt);
		\draw [fill=black] (4.000481566664651,0.8129118959334974) circle (0.5pt);
		\draw [fill=black] (4.015770280755706,0.293095616837617) circle (0.5pt);
	\end{scriptsize}
\end{tikzpicture}

		\caption{ The Graph $K$ }
		\label{newexample}
		
	\end{center}
\end{figure}

\vskip 15 pt

\indent We are ready to establish our next main result. Although the bounds are trivial, we completely characterize all the extremal graphs.
\vskip 5 pt

\begin{them}\label{ob1}
For any non-trivial graph $G$ without any isolated vertex, we have that
$$2 \leq \chi_{i}(G) \leq n.$$
A graph $G$ satisfies the upper bound if and only if $G = K_{n}$. For the lower bound, the extremal graphs $G$ are stars, or otherwise the following statements are equivalent:\\

\indent (i) $\chi_{i}(G) = 2$,\\

\indent (ii) $G$ has Properties $\Im_1$ or $\Im_2$.\\

\indent (iii) $G \in \mathcal{F}_1 \cup \mathcal{F}_2 \cup \mathcal{F}_3$.
\end{them}
\begin{proof}
Since $G$ is non-trivial and has no isolated vertex, it follows that $2 \leq \chi(G) \leq \chi_{i}(G)$, establishing the lower bound. The upper bound is obvious.
\vskip 5 pt

\indent We first characterize the graphs achieving the upper bound. Let $G = K_n$. By Proposition \ref{irp1}, $\chi_{i}(G)=\chi(G)=n$. Conversely, we let $\chi_{i}(G)=\chi(G)=n$. Suppose to the contrary that $G \neq K_n$, then there exists at least two non-adjacent vertices $x,y$ in $G$. Consider the coloring $\mathcal{C}$ of $G$ in which $x,y$ are given same colors and all other vertices are given different colors. Clearly $S=V(G)-\{x,y\}$ is a dominating set of $G$. Let $R$ be the minimal dominating set of $G$ contained in $S$. Then $R$ is maximal irredundant in $G$ ( by Proposition \ref{intro}). Clearly $R$ is $\mathcal{C}$-colorful and hence  $\chi_{i}(G) \leq |\mathcal{C}|=n-1$, contradicting $\chi_{i}(G)=\chi(G)=n$. Thus, $G = K_n$.
\vskip 5 pt

\indent Next, we characterize the graphs achieving the lower bound. We will prove that $(i)$ and $(ii)$ are equivalent. Let $\chi_{i}(G)=2$. Since $\chi_{i}(G) \geq \chi(G)$ implies that $G$ is bipartite with partition $V_1$ and $V_2$. Since $\chi_{i}(G)=2$, there exist $u \in V_1$ and $v \in V_2$ such that $\{u, v\}$ is a maximal irredundant set of $G$. Now we will assume that $G$ does not have Properties $\Im_1$ and $\Im_2$ and arrive at a contradiction. We shall prove that for every $v_1 \in V_1$ and $v_2 \in V_2$, $\{v_1,v_2\}$ is not a maximal irredundant set of $G$, which contradicts our hypothesis. So let us  Assume that $\{v_1,v_2\}$ is a maximal irredundant set of $G$, for some $v_1 \in V_1$ and $v_2 \in V_2$. First we consider the case when $v_1$ is adjacent to $v_2$. Since $G$ does not have Property $\Im_1$, condition (a) or (b) or (c) is not true. Let condition (a) fails. Then there is some $x \in V_2-N(v_1)$ such that $x$ is non-adjacent to some vertex in $N(v_2)-\{v_1\}$. Then the set $\{v_1,v_2,x\}$ is irredundant in $G$, contradicting that $\{v_1,v_2\}$ is maximal irredundant set in $G$. Therefore condition (a) is true. Similarly we arrive at a contradiction if condition (b) fails. Therefore condition (b) is true. Now suppose condition (c) fails, then there exists a vertex $x$, say $x \in N(v_2)$ (where $deg(v_2) \geq 3$) such that condition (c) fails for $x$. Then $deg(x)\geq2$, $N(x) \nsubseteq N(v_1)$ and $N(x) \neq V_2$. This implies that there is some neighbor of $v_1$ which is not adjacent to $x$ (since every non-neighbor of $v_1$ is adjacent to $x$ by condition (a)) and there is some non-neighbor $y$ of $v_1$ in $V_2$ such that $x$ is adjacent to $y$. Then set $\{v_1,v_2,x\}$ is an irredundant set of $G$, a contradiction. Thus condition (c) is also true.  Next we look at the case when $v_1$ and $v_2$ are not adjacent. Since $G$ does not Property $\Im_2$, then there is some $x$ in $V_1$ or $V_2$, say $x \in V_1$ such that $v_2$ is non-adjacent to $x$. Then the set $\{v_1,v_2,x\}$ is irredundant in $G$, contradicting that $\{v_1,v_2\}$ is maximal irredundant set in $G$. Therefore our assumption is wrong and hence $G$ has Properties $\Im_1$ or $\Im_2$.
\vskip 5 pt

\indent Conversely, let $G$ has Properties $\Im_1$ or $\Im_2$. Let $\mathcal{C}=(V_1,V_2)$ be the coloring of $G$. Then there exists vertex $v_1 \in V_1$ and $v_2 \in V_2$ satisfying the conditions of Properties $\Im_1$ or $\Im_2$.
\vskip 5 pt

\begin{cl}\label{claim1}
The set $\{v_1,v_2\}$  is maximal irredundant set of $G$
\end{cl}
\noindent \textbf{Proof of Claim \ref{claim1}}. Suppose to the contrary that $\{v_1,v_2\}$ is not a maximal irredundant set of $G$.
\vskip 5 pt

\indent We first consider the case when $v_1$ is adjacent to $v_2$ satisfying the conditions of the Property $\Im_1$. Since the set $\{v_1,v_2\}$ is not a maximal irredundant, it follows that there is some $x$ (say $x \in V_1$) such that $\{v_1,v_2,x\}$ is an irredundant set of $G$. Let $x$ be non-adjacent to $v_2$. Now by the condition (b), $x$ is adjacent to every vertex in $N(v_1)-\{v_2\}$ which imply that $v_1$ has no private neighbor with respect to the set $\{v_1,v_2,x\}$ contradicting that $\{v_1,v_2,x\}$ is irredundant in $G$. Now suppose $x$ is adjacent to $v_2$. If $deg(v_2)=2$, then $v_2$ has no private neighbor with respect to the set $\{v_1,v_2,x\}$, a contradiction. So let $deg(v_2)\geq 3$. Now by the condition (c), for every $x \in N(v_2)$, either $x$ is pendant or $N(x) \subseteq N(v_1)$ or $N(x)=V_2$. This implies that if $x$ is pendant or $N(x) \subseteq N(v_1)$, then $x$ has no private neighbor with respect to the set $\{v_1,v_2,x\}$ and if $N(x)=V_2$, then $v_1$ has no private neighbor with respect to the set $\{v_1,v_2,x\}$, a contradiction.
\vskip 5 pt

\indent We now consider the case when $v_1$ and $v_2$ are non-adjacent satisfying the condition Property $\Im_2$. Since the set $\{v_1,v_2\}$ is not maximal irredundant, it follows that there is some $x$ (say $x \in V_1$) such that $\{v_1,v_2,x\}$ is an irredundant set of $G$. Since $N(v_1)=V_2-\{v_2\}$ and $N(v_2)=V_1-\{v_1\}$ imply that  $x$ has no private neighbor with respect to the set $\{v_1,v_2,x\}$ contradicting that $\{v_1,v_2,x\}$ is irredundant in $G$. This proves Claim \ref{claim1}.
\vskip 5 pt

\indent Thus $\{v_1,v_2\}$ is maximal irredundant set of $G$ and $\mathcal{C}$-colorful. Therefore $\chi_{i}(G)=2$.
\vskip 5 pt

\indent Finally, we will prove that $(i)$ and $(iii)$ are equivalent. Clearly, all graphs $G$ in the classes $\mathcal{F}_{1}, \mathcal{F}_{2}$ and $\mathcal{F}_{3}$ satisfy $\chi_{i}(G)=2$. Thus, we may let $G$ be a graph with $\chi_{i}(G)=2$. By $(ii)$, $G$ has Properties $\Im_1$ or $\Im_2$. If $G$ has Property $\Im_2$, then  $G \in \mathcal{F}_3$. Now, we may assume that $G$ has Property $\Im_1$. Thus, there exists two adjacent vertices $v_1$ and $v_2$ satisfying the conditions of the Property $\Im_1$. Also $deg(v_i)\geq 2$, for each $i=1,2$. We distinguish $3$ cases.
\vskip 5 pt

\noindent \textbf{Case 1.} $deg(v_1)=2$ and $deg(v_2)=2$.\\
\indent Let $x$ (other than $v_2$) be the neighbor of $v_1$ and $y$ (other than $v_1$) be the neighbor of $v_2$. By Conditions (a) and (b), every non-neighbor $v_1$ is adjacent to $y$ and every non-neighbor $v_2$ is adjacent to $x$. If $x$ is adjacent to $y$, then $G \in \mathcal{F}_2$; otherwise $G \in \mathcal{F}_3$.
\vskip 5 pt

\noindent \textbf{Case 2.} $deg(v_1) \geq 3$ and $deg(v_2)=2$.\\
\indent Let $x$ (other than $v_1$) be the neighbor of $v_2$. By Condition (a), every non-neighbor $v_1$ is adjacent to $x$ and they are pendant. Now some (possible none) or all the vertices of $N(v_1)-\{v_2\}$ are adjacent to $x$. Therefore $G \in \mathcal{F}_1$. Even if $deg(v_2) \geq 3$ and $deg(v_1)=2$, we get a graph belonging to $\mathcal{F}_1$.
\vskip 5 pt	

\noindent \textbf{Case 3.} $deg(v_1) \geq 3$ and $deg(v_2) \geq 3$.\\
\indent Let $x$ (other than $v_2$) be the neighbor of $v_1$ and $y$ (other than $v_1$) be the neighbor of $v_2$. If $N(v_1)=V_2$ and $N(v_2)=V_1$, then $G \in \mathcal{F}_2$. So let $N(v_1) \neq V_2$ and $N(v_2)=V_1$. Now by Condition (a), every non-neighbor of $v_1$ is adjacent to $y$. Now by the Condition (c), $N(y)=V_2$. Therefore there exists vertices $y \in V_1$ and $v_2 \in V_2$ such that $N(y)=V_1$ and $N(v_2)=V_1$. Therefore $G \in \mathcal{F}_2$. Similarly if $N(v_1) \neq V_2$ and $N(v_2) \neq V_1$, we can show that $N(y)=V_1$ and $N(x)=V_1$. Therefore $G \in \mathcal{F}_2$. If $G$ has a full degree vertex, then $G$ is a star.\\
	Conversely, if $G \in \mathcal{F}_1 \cup \mathcal{F}_2 \cup \mathcal{F}_3$, then $\{v_1,v_2\}$ ia a maximal irredundant set of $G$ and is colorful in any $\chi$-coloring of $G$. Thus $\chi_{i}(G)=2$.
	 \qed
\end{proof}

\vskip 5pt
\indent We conclude this section by establishing the realizability of graphs with prescribed irredundance chromatic number.
\vskip 5 pt

\noindent For $k, n \in \mathbb{N}$, we let
\vskip 5 pt

\indent $\mathcal{B}(n, k)$ the family of graphs $G$ of order $n$ such that $\chi_{i}(G)=k$.
\vskip 5 pt

\noindent We prove that the class $\mathcal{B}(n, k)$ is non-empty for any natural numbers $k, n$ such that  $2 \leq k \leq n$.
\vskip 5 pt

\begin{them}\label{thmprescribe}
For $k, n \in \mathbb{N}$ such that  $2 \leq k \leq n$, we have $\mathcal{B}(n, k) \neq \emptyset$.
\end{them}
\begin{proof}
Let $K_k$ be the complete graph with $V(K_k)=\{v_1,v_2, \ldots, v_k\}$. Then, we let $H(n, k)$ be a graph obtained from $K_k$ by joining $n-k$ vertices, namely, $u_1, u_2, \ldots, u_{n-k}$ to the vertex $v_1$. Let $\mathcal{C}$ be the coloring of $H(n, k)$ such that $col(v_i)=i$, for $1 \leq i \leq k$ and $col(u_j)=2$, for $1 \leq j \leq n-k$. Clearly $\chi(H(n, k))=k$ and by Proposition \ref{irp1}, $\{v_1\}$ is a maximal irredundant set of $H(n, k)$. Thus $\chi_{i}(H(n, k))=k$, implying that $H(n, k) \in \mathcal{B}(n, k)$. This proves the theorem.  \qed
\end{proof}
\vskip 5 pt

\indent An example of the graph $H(6, 4) \in \mathcal{B}(6, 4)$ by the construction of Theorem \ref{thmprescribe} is shown in Figure \ref{example3}.
\vskip 5 pt

\begin{figure}[h]
	\begin{center}
		\begin{tikzpicture}[line cap=round,line join=round,>=triangle 45,x=1cm,y=1cm]
			\clip(-4.3,-0.54) rectangle (10,6);
			\draw [line width=0.5pt] (1,3)-- (3,3);
			\draw [line width=0.5pt] (1,3)-- (1,1);
			\draw [line width=0.5pt] (1,3)-- (3,1);
			\draw [line width=0.5pt] (3,1)-- (1,1);
			\draw [line width=0.5pt] (1,1)-- (3,3);
			\draw [line width=0.5pt] (3,3)-- (3,1);
			\draw [line width=0.5pt] (3,3)-- (4,5);
			\draw [line width=0.5pt] (3,3)-- (5,4);
			\begin{scriptsize}
				\draw [fill=black] (1,3) circle (3pt);
				\draw[color=black] (0.5,3.2) node {\large{$v_4$}};
				\draw [fill=black] (1,1) circle (3pt);
				\draw[color=black] (0.5,1.2) node {\large{$v_3$}};
				\draw [fill=black] (3,3) circle (3pt);
				\draw[color=black] (3.4,2.8) node {\large{$v_1$}};
				\draw [fill=black] (3,1) circle (3pt);
				\draw[color=black] (3.4,1.2) node {\large{$v_2$}};
				\draw [fill=black] (4,5) circle (3pt);
				\draw[color=black] (4.22,5.4) node {\large{$u_1$}};
				\draw [fill=black] (5,4) circle (3pt);
				\draw[color=black] (5.22,4.4) node {\large{$u_2$}};
			\end{scriptsize}
		\end{tikzpicture}
		\caption{The graph $H(6, 4)$}
		\label{example3}
	\end{center}
	
\end{figure}
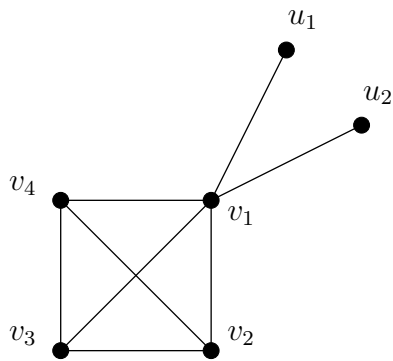




\section{On irredundance compelling coloring of graphs}\label{compelir}
In this section, we present results on irredundance compelling chromatic number of graphs. We also describe some conditions about when a graph admits irredundance compelling coloring.
\vskip 5 pt

\begin{prop}\label{degreemore}
Let $G$ be a $IRC$-colorable graph. Then $\delta(G)\geq 2$.
\end{prop}
\begin{proof}
	Suppose $\delta(G)< 2$ and let $v$ be a pendant vertex, $u$ be the vertex adjacent to $v$. Then the any $RC$ containing $v$ and $u$ is not an irredundant set. Hence $\delta(G) \geq 2$. \qed
\end{proof}
\vskip 5pt

\indent Hence all the graphs considered throughout in this section are of minimum degree at least 2. By Proposition \ref{degreemore}, we have the following corollary.
\vskip 5 pt

\begin{cor}
Non-trivial trees do not admit IRC-coloring.
\end{cor}
\vskip 5pt

\indent Since non-trivial trees are the graphs with edge-connectivity $1$, the question that arises is that: does there exists an $IRC$-colorable graph of edge-connectivity $1$? So in the following, we construct an $IRC$-colorable graphs with vertex-connectivity 1 in Theorem \ref{ttt1}. Then, we proceed to construct an $IRC$-colorable graph with edge-connectivity $1$ in Theorem \ref{cor}.
\vskip 15 pt

\begin{figure}[h]
	\begin{center}
	\begin{tikzpicture}[line cap=round,line join=round,>=triangle 45,x=1cm,y=1cm]
		\clip(-4.3,-0.54) rectangle (11,4.5);
		\draw [line width=0.5pt] (0,4)-- (2,4);
		\draw [line width=0.5pt] (2,4)-- (2,2);
		\draw [line width=0.5pt] (2,2)-- (0,2);
		\draw [line width=0.5pt] (0,2)-- (0,4);
		\draw [line width=0.5pt] (0,2)-- (2,0);
		\draw [line width=0.5pt] (2,0)-- (5,0);
		\draw [line width=0.5pt] (5,0)-- (2,2);
		\draw [line width=0.5pt] (5,4)-- (7,4);
		\draw [line width=0.5pt] (7,4)-- (7,2);
		\draw [line width=0.5pt] (7,2)-- (5,2);
		\draw [line width=0.5pt] (5,2)-- (5,4);
		\draw [line width=0.5pt] (5,2)-- (2,0);
		\draw [line width=0.5pt] (7,2)-- (5,0);
		\begin{scriptsize}
			\draw [fill=black] (2,0) circle (3pt);
			\draw[color=black] (2.1,0.49) node {\large{$u_i$}};
			\draw [fill=black] (5,0) circle (3pt);
			\draw[color=black] (5.1,0.49) node {\large{$v_i$}};
			\draw [fill=black] (0,2) circle (3pt);
			\draw [fill=black] (0,4) circle (3pt);
			\draw [fill=black] (2,4) circle (3pt);
			\draw [fill=black] (2,2) circle (3pt);
			\draw [fill=black] (5,2) circle (3pt);
			\draw [fill=black] (5,4) circle (3pt);
			\draw [fill=black] (7,4) circle (3pt);
			\draw [fill=black] (7,2) circle (3pt);
		\end{scriptsize}
	\end{tikzpicture}
		\caption{ The graph $A^i$ }
		\label{fig1}
	\end{center}
	
\end{figure}
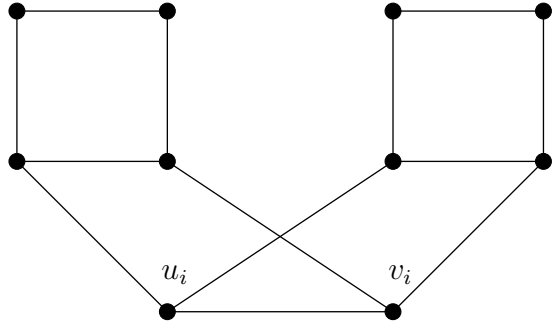
\vskip 5 pt






\begin{them} \label{ttt1}
There are infinitely many $IRC$-colorable graphs with vertex-connectivity 1.
\end{them}
\begin{proof}
We establish this theorem by constructive proof. First, we let $A$ be the graph obtained by a cycle of length $10$, $x_{1}x_{2}...x_{10}x_{1}$ by adding the edges $x_{1}x_{6}, x_{2}x_{5}$ and $x_{7}x_{10}$. In our construction, we may need a number of copies of $A$. For the sake of convenience, we rename the vertices $x_{1}$ and $x_{10}$ of the $i^{th}$ copy $A^{i}$ to be $u_{i}$ and $v_{i}$, respectively. The graph $A^i$ is shown in Figure \ref{fig1}. For a natural number $k \geq 3$, we construct the graph $G(k)$ of vertex-connectivity 1 from $k$ copies of $A$ and a vertex $x$ by adding edges $u_i x$ and $v_i x$ for all $1 \leq i \leq k$. An example of $G(3)$ is shown in Figure \ref{fig2}. Clearly $G(k)$ is a graph of vertex-connectivity 1. By the construction, we have that $\cup^{k}_{i = 1} A^i = G(k) - x$ has the unique coloring (up to isomorphism) $(V_1, V_2)$. Then, it can be checked that the coloring $\mathcal{C}=(V_1,V_2, \{x\})$ is an $IRC$-coloring of $G(k)$. Hence $G(k)$ is $IRC$-colorable. This completes the proof.  \qed
\end{proof}
\vskip 15 pt

\begin{figure}[h!]
	\begin{center}
	\begin{tikzpicture}[line cap=round,line join=round,>=triangle 45,x=1cm,y=1cm]
		\clip(-8,-2.93) rectangle (13,6);
		\draw [line width=0.5pt] (-1,2)-- (1,2);
		\draw [line width=0.5pt] (1,2)-- (0,0);
		\draw [line width=0.5pt] (-1,2)-- (0,0);
		\draw [line width=0.5pt] (0,0)-- (2,1);
		\draw [line width=0.5pt] (2,1)-- (2,-1);
		\draw [line width=0.5pt] (2,-1)-- (0,0);
		\draw [line width=0.5pt] (0,0)-- (-2,-1);
		\draw [line width=0.5pt] (-2,-1)-- (-2,1);
		\draw [line width=0.5pt] (-2,1)-- (0,0);
		\draw [line width=0.5pt] (-2,5)-- (-1,5);
		\draw [line width=0.5pt] (-1,5)-- (-1,4);
		\draw [line width=0.5pt] (-1,4)-- (-2,4);
		\draw [line width=0.5pt] (-2,4)-- (-2,5);
		\draw [line width=0.5pt] (1,5)-- (1,4);
		\draw [line width=0.5pt] (1,4)-- (2,4);
		\draw [line width=0.5pt] (2,4)-- (2,5);
		\draw [line width=0.5pt] (2,5)-- (1,5);
		\draw [line width=0.5pt] (1,4)-- (-1,2);
		\draw [line width=0.5pt] (2,4)-- (1,2);
		\draw [line width=0.5pt] (-2,4)-- (-1,2);
		\draw [line width=0.5pt] (-1,4)-- (1,2);
		\draw [line width=0.5pt] (4,2)-- (4,1);
		\draw [line width=0.5pt] (4,1)-- (5,1);
		\draw [line width=0.5pt] (5,1)-- (5,2);
		\draw [line width=0.5pt] (5,2)-- (4,2);
		\draw [line width=0.5pt] (4,-1)-- (4,-2);
		\draw [line width=0.5pt] (4,-2)-- (5,-2);
		\draw [line width=0.5pt] (5,-2)-- (5,-1);
		\draw [line width=0.5pt] (5,-1)-- (4,-1);
		\draw [line width=0.5pt] (4,-1)-- (2,1);
		\draw [line width=0.5pt] (4,-2)-- (2,-1);
		\draw [line width=0.5pt] (4,1)-- (2,-1);
		\draw [line width=0.5pt] (4,2)-- (2,1);
		\draw [line width=0.5pt] (-5,2)-- (-4,2);
		\draw [line width=0.5pt] (-4,2)-- (-4,1);
		\draw [line width=0.5pt] (-4,1)-- (-5,1);
		\draw [line width=0.5pt] (-5,2)-- (-5,1);
		\draw [line width=0.5pt] (-5,-1)-- (-4,-1);
		\draw [line width=0.5pt] (-4,-1)-- (-4,-2);
		\draw [line width=0.5pt] (-4,-2)-- (-5,-2);
		\draw [line width=0.5pt] (-5,-2)-- (-5,-1);
		\draw [line width=0.5pt] (-4,2)-- (-2,1);
		\draw [line width=0.5pt] (-4,1)-- (-2,-1);
		\draw [line width=0.5pt] (-2,1)-- (-4,-1);
		\draw [line width=0.5pt] (-2,-1)-- (-4,-2);
		\draw (-0.25,-0.3) node[anchor=north west] {\large{$x$}};
		\begin{scriptsize}
			\draw [fill=black] (0,0) circle (3pt);
			\draw [fill=black] (-2,1) circle (3pt);
			\draw[color=black] (-1.78,1.37) node {\large{$u_3$}};
			\draw [fill=black] (-2,-1) circle (3pt);
			\draw[color=black] (-1.78,-1.52) node {\large{$v_3$}};
			\draw [fill=black] (2,1) circle (3pt);
			\draw[color=black] (2.22,1.37) node {\large{$u_2$}};
			\draw [fill=black] (2,-1) circle (3pt);
			\draw[color=black] (2.22,-1.52) node {\large{$v_2$}};
			\draw [fill=black] (-1,2) circle (3pt);
			\draw[color=black] (-1.5,2.2) node {\large{$u_1$}};
			\draw [fill=black] (1,2) circle (3pt);
			\draw[color=black] (1.5,2.2) node {\large{$v_1$}};
			\draw [fill=black] (-1,5) circle (3pt);
			\draw [fill=black] (-1,4) circle (3pt);
			\draw [fill=black] (-2,5) circle (3pt);
			\draw [fill=black] (-2,4) circle (3pt);
			\draw [fill=black] (1,5) circle (3pt);
			\draw [fill=black] (1,4) circle (3pt);
			\draw [fill=black] (2,5) circle (3pt);
			\draw [fill=black] (2,4) circle (3pt);
			\draw [fill=black] (4,1) circle (3pt);
			\draw [fill=black] (5,1) circle (3pt);
			\draw [fill=black] (5,2) circle (3pt);
			\draw [fill=black] (4,2) circle (3pt);
			\draw [fill=black] (4,-1) circle (3pt);
			\draw [fill=black] (5,-1) circle (3pt);
			\draw [fill=black] (5,-2) circle (3pt);
			\draw [fill=black] (4,-2) circle (3pt);
			\draw [fill=black] (-4,1) circle (3pt);
			\draw [fill=black] (-5,1) circle (3pt);
			\draw [fill=black] (-5,2) circle (3pt);
			\draw [fill=black] (-4,2) circle (3pt);
			\draw [fill=black] (-4,-1) circle (3pt);
			\draw [fill=black] (-5,-1) circle (3pt);
			\draw [fill=black] (-5,-2) circle (3pt);
			\draw [fill=black] (-4,-2) circle (3pt);
		\end{scriptsize}
	\end{tikzpicture}
		\caption{ The graph $G(3)$}
		\label{fig2}
	\end{center}
	
\end{figure}
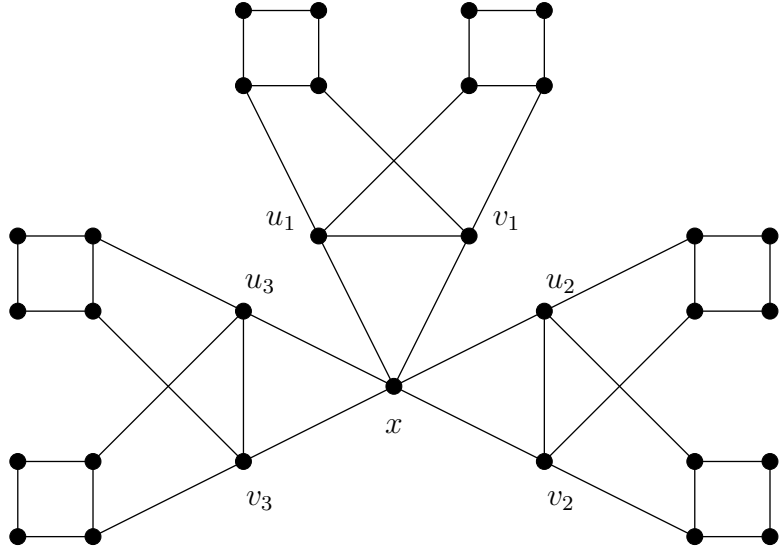
\vskip 15 pt

\begin{them} \label {cor}
		There are infinitely many $IRC$-colorable graphs with edge-connectivity 1.
\end{them}
\begin{proof}
We let $G(k)$ and $G(l)$ be two graphs as constructed in Theorem \ref{ttt1} whose the only cut vertex is $x$ and $y$, respectively. We construct a new graph $G(k, l)$ from $G(k)$ and $G(l)$ by joining $x$ and $y$. Clearly $G(k, l)$ is graph with edge-connectivity 1 in which the edge $xy$ is the bridge. Similarly, we can let $(V_1, V_2)$ be the unique coloring of $G(k, l) - xy$. Thus, the coloring $\mathcal{C}=(V_1,V_2, \{x\}, \{y\})$ is an $IRC$-coloring of $G(k, l)$, completing the proof. \qed
\end{proof}
\vskip 5 pt

\noindent The following is an interesting problem for future work,
\vskip 5 pt

\begin{quest}\label{q1}
	Which graphs admit $IRC$-coloring?
\end{quest}
\vskip 5 pt

\indent With the motive of answering Question \ref{q1}, we provide some conditions for a graph to be $IRC$-colorable as detailed in Propositions \ref{p1} and \ref{p2}.
\vskip 5 pt

\begin{prop}\label{p1}
	Let $G$ is a graph which $IRC$-coloring $\mathcal{C}$. Then for any $v \in V(G)$, at least two vertices in $N(v)$ should have same color in the coloring $\mathcal{C}$.	
\end{prop}
\begin{proof}
Suppose there exists a vertex $v$ such that all the vertices in $N(v)$ receive different colors in an $IRC$-coloring $\mathcal{C}$. Then any $RC$ containing the set $N[v]$ is not an irredundant set which leads to a contradiction. \qed
\end{proof}
\vskip 5 pt

\indent Proposition \ref{p1} implies the following corollary.
\vskip 5 pt

\begin{cor}
All the cycles of odd order are not $IRC$-colorable.
\end{cor}
\begin{prop}\label{p2}
	Let $Q$ be a clique in $G$. If there exists a vertex $v$ in $Q$ such that $pn[v,Q]=\emptyset$, then $G$ is not $IRC$-colorable.
\end{prop}
\begin{proof}
	Since $Q$ is a clique, vertices of $Q$ receive unique colors. Suppose there exists a vertex $v$ in $Q$ such that $pn[v,Q]=\emptyset$, then any $RC$ containing $Q$ is not an irredundant set. Therefore $G$ is not $IRC$-colorable. \qed
\end{proof}
\vskip 5 pt

\begin{cor}
	Split graphs $G$ with minimum degree at least two are not $IRC$-colorable.
\end{cor}
\begin{proof}
	Let $Q$ be a clique of maximum order in $G$. Let $v \in V(Q)$. Clearly $pn[v,Q]=\emptyset$. By Proposition \ref{p2} $G$ is not $IRC$-colorable. \qed
\end{proof}
\vskip 15 pt

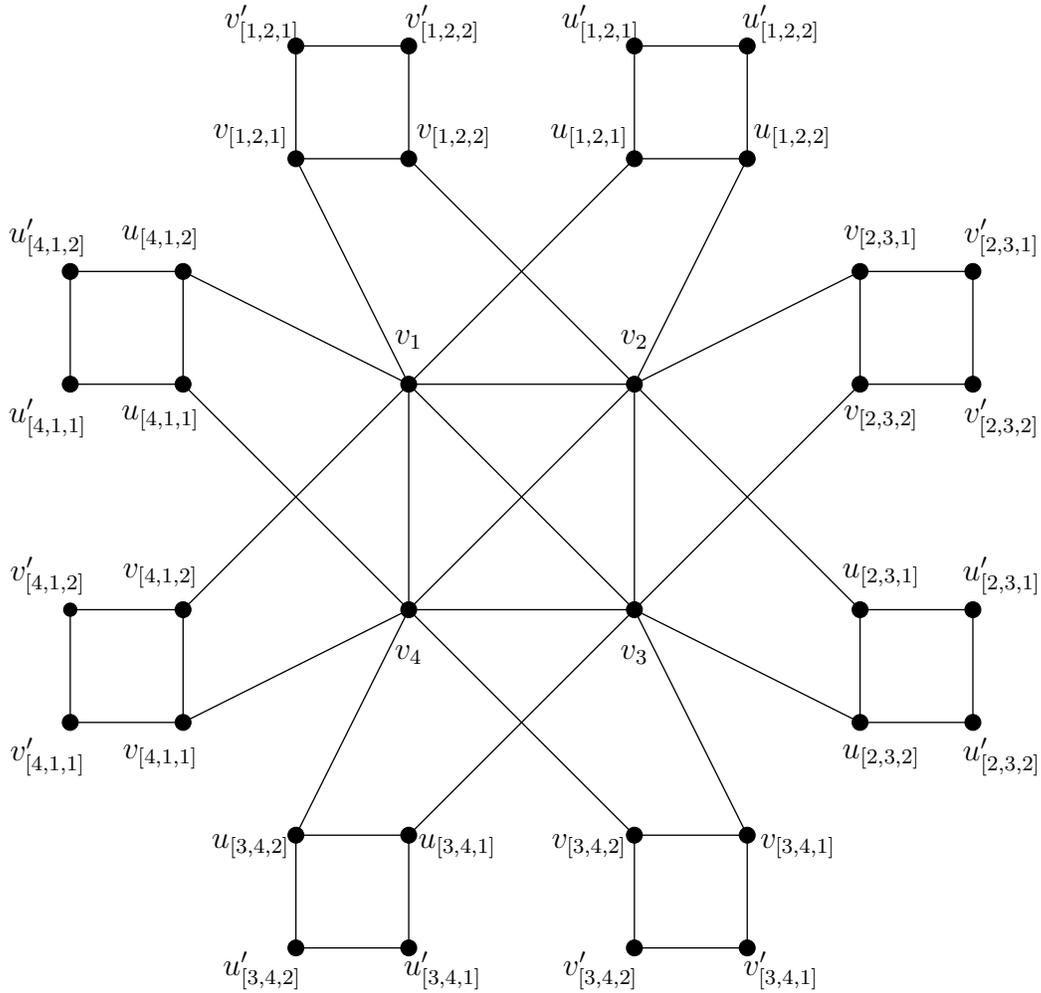
\begin{figure}[h!]
	\begin{center}
		\begin{tikzpicture}[line cap=round,line join=round,>=triangle 45,x=1.5cm,y=1.5cm]
			\clip(-5.207036668127865,-4.707583213550753) rectangle (5,5);
			\draw [line width=0.5pt] (-1,1)-- (1,1);
			\draw [line width=0.5pt] (1,1)-- (1,-1);
			\draw [line width=0.5pt] (1,-1)-- (-1,-1);
			\draw [line width=0.5pt] (-1,-1)-- (-1,1);
			\draw [line width=0.5pt] (-1,1)-- (1,-1);
			\draw [line width=0.5pt] (1,1)-- (-1,-1);
			\draw [line width=0.5pt] (-2,-3)-- (-2,-4);
			\draw [line width=0.5pt] (-2,-4)-- (-1,-4);
			\draw [line width=0.5pt] (-1,-4)-- (-1,-3);
			\draw [line width=0.5pt] (-1,-3)-- (-2,-3);
			\draw [line width=0.5pt] (1,-3)-- (1,-4);
			\draw [line width=0.5pt] (1,-4)-- (2,-4);
			\draw [line width=0.5pt] (2,-4)-- (2,-3);
			\draw [line width=0.5pt] (2,-3)-- (1,-3);
			\draw [line width=0.5pt] (-1,-1)-- (-2,-3);
			\draw [line width=0.5pt] (1,-1)-- (-1,-3);
			\draw [line width=0.5pt] (1,-1)-- (2,-3);
			\draw [line width=0.5pt] (-1,-1)-- (1,-3);
			\draw [line width=0.5pt] (3,2)-- (4,2);
			\draw [line width=0.5pt] (4,2)-- (4,1);
			\draw [line width=0.5pt] (4,1)-- (3,1);
			\draw [line width=0.5pt] (3,2)-- (3,1);
			\draw [line width=0.5pt] (3,-1)-- (3,-2);
			\draw [line width=0.5pt] (3,-2)-- (4,-2);
			\draw [line width=0.5pt] (4,-2)-- (4,-1);
			\draw [line width=0.5pt] (4,-1)-- (3,-1);
			\draw [line width=0.5pt] (3,2)-- (1,1);
			\draw [line width=0.5pt] (3,1)-- (1,-1);
			\draw [line width=0.5pt] (3,-1)-- (1,1);
			\draw [line width=0.5pt] (3,-2)-- (1,-1);
			\draw [line width=0.5pt] (-4,2)-- (-4,1);
			\draw [line width=0.5pt] (-4,1)-- (-3,1);
			\draw [line width=0.5pt] (-3,1)-- (-3,2);
			\draw [line width=0.5pt] (-3,2)-- (-4,2);
			\draw [line width=0.5pt] (-4,-1)-- (-4,-2);
			\draw [line width=0.5pt] (-4,-2)-- (-3,-2);
			\draw [line width=0.5pt] (-3,-2)-- (-3,-1);
			\draw [line width=0.5pt] (-3,-1)-- (-4,-1);
			\draw [line width=0.5pt] (-3,2)-- (-1,1);
			\draw [line width=0.5pt] (-3,1)-- (-1,-1);
			\draw [line width=0.5pt] (-1,1)-- (-3,-1);
			\draw [line width=0.5pt] (-1,-1)-- (-3,-2);
			\draw [line width=0.5pt] (-2,4)-- (-2,3);
			\draw [line width=0.5pt] (-2,3)-- (-1,3);
			\draw [line width=0.5pt] (-1,3)-- (-1,4);
			\draw [line width=0.5pt] (-1,4)-- (-2,4);
			\draw [line width=0.5pt] (1,4)-- (1,3);
			\draw [line width=0.5pt] (1,3)-- (2,3);
			\draw [line width=0.5pt] (2,3)-- (2,4);
			\draw [line width=0.5pt] (2,4)-- (1,4);
			\draw [line width=0.5pt] (-2,3)-- (-1,1);
			\draw [line width=0.5pt] (-1,3)-- (1,1);
			\draw [line width=0.5pt] (1,3)-- (-1,1);
			\draw [line width=0.5pt] (2,3)-- (1,1);
			\begin{scriptsize}
				\draw [fill=black] (-1,1) circle (3pt);
				\draw[color=black] (-1,1.4) node {\large{$v_{1}$}};
				\draw [fill=black] (1,1) circle (3pt);
				\draw[color=black] (1,1.4) node {\large{$v_2$}};
				\draw [fill=black] (1,-1) circle (3pt);
				\draw[color=black] (1,-1.4) node {\large{$v_3$}};
				\draw [fill=black] (-1,-1) circle (3pt);
				\draw[color=black] (-1,-1.4) node {\large{$v_4$}};
				\draw [fill=black] (-2,4) circle (3pt);
				\draw[color=black] (-2.3,4.2) node {\large{$v_{[1,2,1]}^{\prime}$}};
				\draw [fill=black] (-1,4) circle (3pt);
				\draw[color=black] (-0.7,4.2) node {\large{$v_{[1,2,2]}^{\prime}$}};
				\draw [fill=black] (-1,3) circle (3pt);
				\draw[color=black] (-0.6,3.2) node {\large{$v_{[1,2,2]}$}};
				\draw [fill=black] (-2,3) circle (3pt);
				\draw[color=black] (-2.4,3.2) node {\large{$v_{[1,2,1]}$}};
				\draw [fill=black] (1,4) circle (3pt);
				\draw[color=black] (0.7,4.2) node {\large{$u_{[1,2,1]}^{\prime}$}};
				\draw [fill=black] (1,3) circle (3pt);
				\draw[color=black] (0.6,3.2) node {\large{$u_{[1,2,1]}$}};
				\draw [fill=black] (2,4) circle (3pt);
				\draw[color=black] (2.3,4.2) node {\large{$u_{[1,2,2]}^{\prime}$}};
				\draw [fill=black] (2,3) circle (3pt);
				\draw[color=black] (2.4,3.2) node {\large{$u_{[1,2,2]}$}};
				\draw [fill=black] (3,1) circle (3pt);
				\draw[color=black] (3.1873446971559383,0.7) node {\large{$v_{[2,3,2]}$}};
				\draw [fill=black] (4,1) circle (3pt);
				\draw[color=black] (4.255609013653294,0.7) node {\large{$v_{[2,3,2]}^{\prime}$}};
				\draw [fill=black] (3,2) circle (3pt);
				\draw[color=black] (3.1873446971559383,2.3) node {\large{$v_{[2,3,1]}$}};
				\draw [fill=black] (4,2) circle (3pt);
				\draw[color=black] (4.255609013653294,2.3) node {\large{$v_{[2,3,1]}^{\prime}$}};
				\draw [fill=black] (3,-1) circle (3pt);
				\draw[color=black] (3.1873446971559383,-0.7) node {\large{$u_{[2,3,1]}$}};
				\draw [fill=black] (4,-1) circle (3pt);
				\draw[color=black] (4.255609013653294,-0.7) node {\large{$u_{[2,3,1]}^{\prime}$}};
				\draw [fill=black] (3,-2) circle (3pt);
				\draw[color=black] (3.1873446971559383,-2.3) node {\large{$u_{[2,3,2]}$}};
				\draw [fill=black] (4,-2) circle (3pt);
				\draw[color=black] (4.255609013653294,-2.3) node {\large{$u_{[2,3,2]}^{\prime}$}};
				\draw [fill=black] (1,-3) circle (3pt);
				\draw[color=black] (0.6,-3.1) node {\large{$v_{[3,4,2]}$}};
				\draw [fill=black] (2,-3) circle (3pt);
				\draw[color=black] (2.45,-3.1) node {\large{$v_{[3,4,1]}$}};
				\draw [fill=black] (2,-4) circle (3pt);
				\draw[color=black] (2.3,-4.2) node {\large{$v_{[3,4,1]}^{\prime}$}};
				\draw [fill=black] (1,-4) circle (3pt);
				\draw[color=black] (0.7,-4.2) node {\large{$v_{[3,4,2]}^{\prime}$}};
				\draw [fill=black] (-1,-3) circle (3pt);
				\draw[color=black] (-0.57,-3.1) node {\large{$u_{[3,4,1]}$}};
				\draw [fill=black] (-1,-4) circle (3pt);
				\draw[color=black] (-0.7,-4.2) node {\large{$u_{[3,4,1]}^{\prime}$}};
				\draw [fill=black] (-2,-3) circle (3pt);
				\draw[color=black] (-2.4,-3.1) node {\large{$u_{[3,4,2]}$}};
				\draw [fill=black] (-2,-4) circle (3pt);
				\draw[color=black] (-2.3,-4.2) node {\large{$u_{[3,4,2]}^{\prime}$}};
				\draw [fill=black] (-3,-1) circle (3pt);
				\draw[color=black] (-3.2,-0.7) node {\large{$v_{[4,1,2]}$}};
				\draw [fill=black] (-4,-1) circle (2.5pt);
				\draw[color=black] (-4.2,-0.7) node {\large{$v_{[4,1,2]}^{\prime}$}};
				\draw [fill=black] (-4,-2) circle (3pt);
				\draw[color=black] (-4.2,-2.3) node {\large{$v_{[4,1,1]}^{\prime}$}};
				\draw [fill=black] (-3,-2) circle (3pt);
				\draw[color=black] (-3.2,-2.3) node {\large{$v_{[4,1,1]}$}};
				\draw [fill=black] (-3,1) circle (3pt);
				\draw[color=black] (-3.2,0.7) node {\large{$u_{[4,1,1]}$}};
				\draw [fill=black] (-4,1) circle (3pt);
				\draw[color=black] (-4.2,0.7) node {\large{$u_{[4,1,1]}^{\prime}$}};
				\draw [fill=black] (-4,2) circle (3pt);
				\draw[color=black] (-4.2,2.3) node {\large{$u_{[4,1,2]}^{\prime}$}};
				\draw [fill=black] (-3,2) circle (3pt);
				\draw[color=black] (-3.2,2.3) node {\large{$u_{[4,1,2]}$}};
			\end{scriptsize}
		\end{tikzpicture}
		\caption{ The graph $\tilde{G}(4)$ in the class $\mathcal{G}(4)$}
		\label{example1}
	\end{center}
\end{figure}
\vskip 15 pt

\indent Although not every graph posses an $IRC$-coloring, we can establish the existence of graphs $G$ when $\chi_{irc}(G)$ is given.
\vskip 5 pt

\noindent For any natural number $k \geq 2$, we let
\vskip 5 pt

\indent $\mathcal{G}(k)$ be the family of graphs $G$ in which $G$ is $IRC$-colorable and $\chi_{irc}(G) = k$.
\vskip 5 pt

\noindent The following theorem shows that the family $\mathcal{G}$ is non-empty for all $k$.
\vskip 5 pt

\begin{them}\label{them1}
For every natural number $k \geq 2$, we have that $\mathcal{G}(k) \neq \emptyset$.
\end{them}
\begin{proof}
When $k=2$, we have $\chi_{irc}(C_4)=2$ implying that $C_{4} \in \mathcal{G}(k)$. We now proceed the case when $k \geq 3$. Let $K_k$ be the complete graph with vertex set $\{v_1, v_2, \ldots, v_k\}$. The graph $\tilde{G}(k)$ is constructed from $K_k$ and the vertices $v_{[i,i+1,j]}, u_{[i,i+1,j]}, v_{[i,i+1,j]}^{\prime}, u_{[i,i+1,j]}^{\prime}$ for all $1 \leq i \leq k$ and $1 \leq j \leq 2$ by adding edges as follows:
\begin{itemize}
  \item Add the edges  $v_{[i,i+1,1]}v_{[i,i+1,2]}$, $v_{[i,i+1,1]}^{\prime} v_{[i,i+1,2]}^{\prime}$, $v_{[i,i+1,1]} v_{[i,i+1,1]}^{\prime}$, $v_{[i,i+1,2]}v_{[i,i+1,2]}^{\prime}$.
  \item Add the edges  $u_{[i,i+1,1]}u_{[i,i+1,2]}$, $u_{[i,i+1,1]}^{\prime} u_{[i,i+1,2]}^{\prime}$, $u_{[i,i+1,1]} u_{[i,i+1,1]}^{\prime}$, $u_{[i,i+1,2]}u_{[i,i+1,2]}^{\prime}$.
  \item Join the vertex $v_i$ to $v_{[i,i+1,1]}$ and $u_{[i,i+1,1]}$ and join the vertex $v_{i+1}$ to $v_{[i,i+1,2]}$ and $u_{[i,i+1,2]}$.
\end{itemize}
We note that $v_{i}v_{j} \in E(\tilde{G}(k))$ for every $1 \leq i,j \leq k$ because $\tilde{G}(k)$ is constructed from $K_{k}$. We note also that if $i=k$, then $i+1=1$. The construction of $G$ for $k=4$ is illustrated in Figure \ref{example1}. Color the vertices of $\tilde{G}(k)$ in such a way that the vertices $v_i$ are given unique colors such that $col(v_i)=i$. Let $col(v_{[i,i+1,1]})=col(u_{[i,i+1,1]})=col(v_{[i,i+1,2]}^{\prime})=col(u_{[i,i+1,2]}^{\prime})=i+1$ and $col(v_{[i,i+1,2]})=col(u_{[i,i+1,2]})=col(v_{[i,i+1,1]}^{\prime})=col(u_{[i,i+1,1]}^{\prime})=i$. The coloring described with $k$ number of colors is an $IRC$-coloring $\tilde{G}(k)$ and hence $\chi_{irc}(\tilde{G}(k)) \geq k$. \\
\vskip 5 pt

\indent We will show that it is not possible to color with at least $k+1$ number of colors that admit $IRC$-coloring of $\tilde{G}(k)$. Assume to the contrary that there exists an $IRC$-coloring $\mathcal{C}$ of $\tilde{G}(k)$ with at least $k+1$ colors. Renaming the vertices if necessary, we assume that $col(v_i)=i$, for $1 \leq i \leq k$. Thus, by symmetry, we let $col(v_{[1,2,1]})=k+1$. By Preposition \ref{p1}, $col(v_{[1,2,2]}^{\prime})=k+1$ because the only neighbors of $v_{[1,2,1]}^{\prime}$ are $v_{[1,2,1]}$ and $v_{[1,2,2]}^{\prime}$. Since $v_{[1,2,2]}v_2 \in E(\tilde{G}(k))$, $col(v_{[1,2,2]})=j \notin \{2,k+1\}$. Again by Proposition \ref{p1}, $col(v_{[1,2,2]})=col(v_{[1,2,1]}^{\prime})=j \notin \{2,k+1\}$. Clearly the $RC$ containing $v_{[1,2,1]}^{\prime}, v_{[1,2,1]}, v_2$ is not an irredundant set of $\tilde{G}(k)$ as $v_{[1,2,1]}$ has no private neighbor, contradicting that $\mathcal{C}$ is an $IRC$-coloring of $\tilde{G}(k)$. Thus $\chi_{irc}(G) \leq k$.
\vskip 5 pt

\indent Therefore, $\chi_{irc}(\tilde{G}(k)) = k$ implying that $\tilde{G}(k) \in \mathcal{G}(k)$. This completes the proof.\qed
\end{proof}
\vskip 15 pt

\begin{figure}[h!]
	\begin{center}
		\begin{tikzpicture}[line cap=round,line join=round,>=triangle 45,x=1cm,y=1cm]
			\clip(-2.3,-0.5) rectangle (5,5);
			\draw [line width=0.5pt] (0,0)-- (3,0);
			\draw [line width=0.5pt] (3,1)-- (0,0);
			\draw [line width=0.5pt] (3,2)-- (0,0);
			\draw [line width=0.5pt] (0,1)-- (3,4);
			\draw [line width=0.5pt] (3,3)-- (0,1);
			\draw [line width=0.5pt] (3,2)-- (0,1);
			\draw [line width=0.5pt] (0,1)-- (3,0);
			\draw [line width=0.5pt] (0,2)-- (3,4);
			\draw [line width=0.5pt] (3,3)-- (0,2);
			\draw [line width=0.5pt] (0,2)-- (3,1);
			\draw [line width=0.5pt] (0,2)-- (3,0);
			\draw [line width=0.5pt] (0,3)-- (3,4);
			\draw [line width=0.5pt] (0,3)-- (3,3);
			\draw [line width=0.5pt] (0,3)-- (3,2);
			\draw [line width=0.5pt] (0,3)-- (3,1);
			\begin{scriptsize}
				\draw [fill=black] (0,0) circle (3pt);
				\draw[color=black] (-0.5,0) node {\large{$v^{*}$}};
				\draw [fill=black] (0,1) circle (3pt);
				\draw [fill=black] (0,2) circle (3pt);
				\draw [fill=black] (0,3) circle (3pt);
				\draw [fill=black] (3,0) circle (3pt);
				\draw [fill=black] (3,1) circle (3pt);
				\draw [fill=black] (3,2) circle (3pt);
				\draw [fill=black] (3,3) circle (3pt);
				\draw [fill=black] (3,4) circle (3pt);
			\end{scriptsize}
		\end{tikzpicture}
		\caption{ The graph $G$ belonging to the family $\mathcal{H}$ }
		\label{eexample2}
	\end{center}
	
\end{figure}
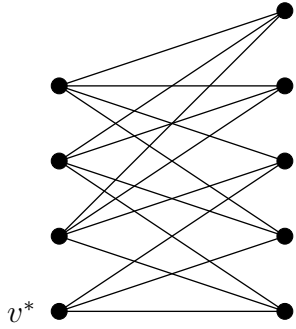
\vskip 15 pt

\indent In our study, we also establish some results of $IRC$-coloring when the graph is bipartite. We define a family $\mathcal{H}$ of bipartite graphs as follows. A bipartite graph $G$ with partite sets $V_1$ and $V_2$ belongs to the family $\mathcal{H}$ if there exists a vertex $v^*$ belonging to either $V_1$ or $V_2$ (say $v^* \in V_1$) satisfying the following property:
\vskip 5 pt

\indent \emph{(*) for every $v \in V_1 - \{v^*\}$, both $v$ and $v^*$ have at least two external private neighbors with respect to the set $\{v,v^*\}$.}
\vskip 5 pt

\noindent An example of a graph in $\mathcal{H}$ is shown in Figure \ref{eexample2}. We prove that all graphs $G$ in this class satisfy $\chi_{irc}(G)\geq 3$.
\vskip 5 pt

\begin{prop}
	If $G \in \mathcal{H}$, then $\chi_{irc}(G)\geq 3$.
\end{prop}
\begin{proof}
	Let $G$ be a  bipartite with partite sets $V_1$, $V_2$. Since $G \in \mathcal{H}$, there exists a vertex $v^*$, say $v^* \in V_1$ satisfying the Property $(*)$ of the graph family $\mathcal{H}$. It can be seen that the coloring $\mathcal{C}=(\{v^*\}, V_{1}-\{v^*\}, V_2)$ is an $IRC$-coloring of $G$ and hence $\chi_{irc}(G) \geq 3$.  \qed
\end{proof}

\indent Interestingly, we can always find $IRC$-coloring graphs with arbitrary large even $IRC$-number although the graphs are bipartite.
\vskip 5 pt

\noindent For an even number $k \geq 2$, we  let
\vskip 5 pt

\indent $\mathcal{R}(k)$ the class of bipartite graphs $G$ such that $\chi_{irc}(G) = k$.
\vskip 5 pt

\noindent The following proposition establish a construction of graphs that are in the class $\mathcal{R}(k)$
\vskip 5 pt

\begin{prop}\label{propp}
For an even number $k \geq 2$, we have that $\mathcal{R}(k) \neq \emptyset$.
\end{prop}
\begin{proof}
When $k = 2$, clearly, $C_{4} \in \mathcal{R}(2)$. We may assume that $k \geq 4$. We let $\tilde{G}(k)$ be the graph which was constructed in Theorem \ref{them1}. The graph $G^*(k)$ is obtained from $\tilde{G}(k)$ by replacing $K_k$ by a cycle $C_k$. It can be checked that he graph $G^*(k)$ is bipartite. Then, we color $G^*(k)$ the same way as that of $\tilde{G}(k)$ and this is also an $IRC$-coloring of $G^*(k)$. Therefore $\chi_{irc}(G^*(k))\geq k$.
\end{proof}
\vskip 5 pt

\indent Finally, we finish this section by some a sufficient condition related with dominator coloring and domination number to confirm the existence of an $IRC$-coloring of a graph.
\vskip 5 pt

\begin{prop}
	If $G$ is a graph with $\chi_{d}(G)= \gamma (G)$, then $G$ is $IRC$-colorable and $\chi_{irc}(G) \geq \gamma (G)$.
\end{prop}
\begin{proof}
	Let $\mathcal{C}$ be the dominator coloring of $G$ with $\chi_{d}(G)= \gamma(G)$. Then every RC with respect to $\mathcal{C}$ is a minimum dominating set of $G$. Since every minimum dominating set of $G$ is an irredundant set of $G$ {(by Proposition \ref{intro})} implies that $\mathcal{C}$ is an $IRC$-coloring of $G$ and hence $\chi_{irc}(G) \geq \gamma (G)$. \qed
	
\end{proof}

\section{Open Problems}\label{problem}
The following are some problems on irredundance chromatic number of a graph. Due to the results of Theorems \ref{thmchi} and \ref{ob1}, it is still possible to obtain the structures of graphs that are close to the bounds. We rise the problem that:
\vskip 5 pt

\begin{problem}
	Characterize graphs $G$ with $\chi_{i}(G)=3$.
\end{problem}
\begin{problem}
	Characterize graphs $G$ with $\chi_{i}(G)=n-1$.
\end{problem}
\begin{problem}
	Characterize graphs $G$ with $\chi_{i}(G)=\chi(G)$.
	\begin{problem}
		Characterize graphs $G$ with $\chi_{i}(G)= ir(G)$.
	\end{problem}
\end{problem}
\vskip 5 pt

\indent In the context of $IRC$-colorable graphs, it would be interesting to establish the necessary and sufficient condition of graphs that posses an $IRC$-coloring.
\vskip 5 pt

\begin{problem}
	Find necessary and sufficient condition for a graph to be $IRC$-colorable.
\end{problem}
\vskip 5 pt

\indent We believe that the following problems are still possible to solve.
\vskip 5 pt

\begin{problem}
	Characterize graphs $G$ with $\chi_{irc}(G)=2$.
\end{problem}
\begin{problem}
	Find an upper bound for $\chi_{irc}(G)$.
\end{problem}
\vskip 5 pt

\indent According to our collection of graphs that posses $IRC$-coloring, the graphs in Proposition \ref{propp} for example, once we obtain an graph $G$ that admit $IRC$-coloring with $\ell$ colors, we can always reduce the number of color to $\chi(G)$. We also believe that this is true for any $IRC$-colorable graphs. We then conjecture that:
\vskip 5 pt

\begin{conj}\label{con1}
	Any $IRC$-colorable graph $G$ admits an $IRC$-coloring using $\chi(G)$ number of colors.
\end{conj}

\section*{Acknowledgements}
The first author acknowledges that this research was supported by King
Mongkut's University of Technology Thonburi Postdoctoral Fellowship.

\end{document}